\documentclass[A4wide,10pt]{amsart}

\usepackage{kotex}
\usepackage{float}
\usepackage{amsmath,amsthm,amscd,amsfonts,amssymb,epic,eepic,bbm,graphicx,ytableau}
\usepackage{tikz}
\usepackage{tkz-euclide}
\usepackage{rotating}
\usepackage{subcaption}
\usepackage{hyperref}
\usetikzlibrary{angles,quotes}

\hypersetup{colorlinks}
\hypersetup{
	unicode=false,          
	pdftoolbar=true,        
	pdfmenubar=true,        
	pdffitwindow=false,     
	pdfstartview={FitH},    
	pdftitle={My title},    
	pdfauthor={Author},     
	pdfsubject={Subject},   
	pdfcreator={Creator},   
	pdfproducer={Producer}, 
	pdfkeywords={keyword1} {key2} {key3}, 
	pdfnewwindow=true,      
	colorlinks=true,       
	linkcolor=blue,          
	citecolor=red,        
	filecolor=violet,      
	urlcolor=violet       
}

\usepackage{tikz-cd}
\usepackage{tikz}
\usetikzlibrary{snakes, 
	3d, matrix,decorations.pathreplacing,calc,decorations.pathmorphing}
\usetikzlibrary {positioning}
\usetikzlibrary{matrix, fit}
\usetikzlibrary{backgrounds}

\newtheorem{thm}{Theorem}[section]
\newtheorem{cor}[thm]{Corollary}

\newtheorem{lem}[thm]{Lemma}
\newtheorem{defn}[thm]{Definition}

\newtheorem{prop}[thm]{Proposition}

\newtheorem{rem}[thm]{Remark}
\newtheorem{ex}[thm]{Example}
\newtheorem{problem}[thm]{Problem}

\renewcommand{\tilde}{\widetilde}

\newcommand{\C}{\mathbb{C}}

\newcommand{\R}{\mathbb{R}}

\newcommand{\Z}{\mathbb{Z}}

\newcommand{\B}{\mathcal{B}}

\DeclareMathOperator{\Cone}{Cone}

\newcommand{\GL}{\operatorname{GL}}
\newcommand{\T}{\operatorname{t}}

\numberwithin{equation}{section}

\title{Toric varieties of Schr\"{o}der type}

\author[JiSun Huh]{JiSun Huh}
\address[J. Huh]{Department of Mathematics, Ajou University, Suwon 16499, Republic of Korea}
\email{hyunyjia@ajou.ac.kr}

\author[Seonjeong Park]{Seonjeong Park}
\address[S. Park]{Department of Mathematics Education, Jeonju University, Jeonju 55069, Republic of Korea}
\email{seonjeongpark@jj.ac.kr}

\thanks{This work was supported by NRF-2020R1A2C1A01011045 and  NRF-2021R1A6A1A10044950}

\date{\today}


\begin{document}

\maketitle

\begin{abstract}
    A dissection of a polygon is obtained by drawing diagonals such that no two diagonals intersect in their interiors. In this paper, we define a toric variety of Schr\"{o}der type as a smooth toric variety associated with a polygon dissection. Toric varieties of Schr\"{o}der type are Fano generalized Bott manifolds, and they are isomorphic if and only if the associated Schr\"{o}der trees are the same as unordered rooted trees. We describe the cohomology ring of a toric variety of Schr\"{o}der type using the associated Schr\"{o}der tree and discuss the cohomological rigidity problem.
\end{abstract}

\section{Introduction}

Let $\mathsf{P}_{n+2}$ be a regular polygon with $n+2$ vertices.
A \emph{dissection} $D$ of $\mathsf{P}_{n+2}$ is obtained by drawing diagonals such that no two diagonals intersect in their interiors. 
A dissection $D$ of $\mathsf{P}_{n+2}$ is called a \emph{$k$-dissection} if it divides $\mathsf{P}_{n+2}$ into $k$  polygons of smaller size. Note that $k$ varies from $1$ to $n$, and an $n$-dissection of $\mathsf{P}_{n+2}$ is called a \emph{triangulation}.

Let \( f(n,k) \) be the number of \( k \)-dissections of \( \mathsf{P}_{n+2}\) for  \( 1\leq k \leq n \). Cayley \cite{Cayley} proved that  \( f(n,k)=\frac{1}{k} \binom{n-1}{k-1} \binom{n+k}{k-1} \), and the number \( f(n,k) \) is called  the
Kirkman-Cayley number. There are a number of equivalent ways to represent polygon dissections: as binary bracketings, \L ukasiewicz words, standard Young tableaux of shape \( (k,k,1^{n-k}) \), and rooted plane trees whose non-leaf vertices have at least two children. (See \cite{Ether, Cayley59, Luka, Schroder, Stanley96}.) The last one is called a \emph{Schr\"{o}der tree} because the total number of these trees with \( n+1 \) leaves, \(\sum_{1\leq k \leq n} f(n,k) \), is the \( (n+1) \)st \emph{small Schr\"{o}der number} \( s_{n+1} \). 
Note that \( f(n, n) \) is just the Catalan number \( C_{n} \), and the Schr\"{o}der trees corresponding to triangulations of a polygon are full binary rooted trees. 

A toric variety of (complex) dimension~$n$ is a normal complex algebraic variety containing a torus $T=(\C^\ast)^n$ as an open dense subset, such that the action of $T$ on itself extends to the whole variety. A typical example of a smooth compact toric variety is the projective space $\C P^n$ of complex dimension $n$ with the standard action of $T$.

We construct a smooth toric variety $X_D$ associated with a dissection $D$ of the polygon $\mathsf{P}_{n+2}$ inductively. We first associate the projective space $\mathbb{C}P^n$ with the polygon $\mathsf{P}_{n+2}$. If $\widetilde{D}$ is a dissection of $\mathsf{P}_{n+2}$ obtained from $D$ by drawing a new diagonal, then the toric variety $X_{\widetilde{D}}$ associated with $\widetilde{D}$ is the blow up of $X_D$ along a certain $T$-invariant submanifold, see Section~\ref{sec:construction} for the precise definition. We say that $X_D$ is \emph{of Schr\"{o}der type}. 

Recall that a toric variety of Catalan type introduced in~\cite{LMP_Catalan} is the toric variety associated with a triangulation of a polygon. When $D$ is a triangulation of a polygon, the fan $\Sigma_D$ in our construction is the same as that in~\cite{LMP_Catalan}, so our toric varieties of Schr\"{o}der type contain the toric varieties of Catalan type.

Every toric variety of Schr\"{o}der type is a generalized Bott manifold, the total space of an iterated complex projective space bundle over a point. If $D$ is a $k$-dissection of $\mathsf{P}_{n+2}$, then $X_D$ is a $k$-stage generalized Bott manifold of dimension~$n$. Furthermore, each toric variety of Schr\"{o}der type is Fano.

Let $\tau(D)$ be the Schr\"{o}der tree associated with a polygon dissection~$D$. Using the Fano condition of $X_D$, we can recover the Schr\"{o}der tree $\tau(D)$ and prove the following.  

\begin{thm}[Theorem~\ref{thm:Fano} and Corollary~\ref{cor:number}]
For two dissections $D$ and $\widetilde{D}$,
the toric varieties $X_D$ and $X_{\widetilde{D}}$ of Schr\"{o}der type are isomorphic as varieties if and only if the associated Schr\"{o}der trees $\tau(D)$ and $\tau(\widetilde{D})$ are isomorphic as unodered rooted trees, and the number of isomorphism classes of $n$-dimensional toric varieties of Schr\"{o}der type is equal to the number of series-reduced rooted trees with $n+1$ leaves.
\end{thm}

Cohomological rigidity problem for toric varieties asks whether two smooth compact toric varieties are diffeomorphic when their cohomology rings are isomorphic as graded rings, see~\cite{Masuda-Suh}. This problem is still open even for generalized Bott manifolds and many results have been produced in support of the affirmative answer to the problem, see~\cite{Choi, Choi-Masuda, CMS10_trans, HKMP}.

 It is proved in~\cite{LMP_directed_Dynkin} that two toric varieties of Catalan type are isomorphic as varieties if and only if their integral cohomology rings are isomorphic as graded rings. Hence we may ask the following.

\begin{problem}\label{prob1}
When two toric varieties $X_D$ and $X_{\widetilde{D}}$ of Schr\"{o}der type have the same cohomology ring, are they isomorphic?
\end{problem}

We explicitly describe the integral cohomology ring of $X_D$ from the associated Schr\"{o}der tree $\tau(D)$, see Theorem~\ref{thm:main} for more details. We give a partial affirmative answer to the above problem when we restrict the number of partitions for polygon dissections.

\begin{thm}\label{thm1}
    Let $D$ and $\widetilde{D}$ be $k$-dissections of the polygon $\mathsf{P}_{n+2}$. When $k\leq 3$ or $k=n$, toric varieties $X_D$ and $X_{\widetilde{D}}$ are isomorphic as varieties if and only if $H^\ast(X_D)$ and $H^\ast(X_{\widetilde{D}})$ are isomorphic as graded rings.
\end{thm}

This paper is organized as follows. In Section~\ref{sec:preliminaries}, we introduce the Etherington's bijection between polygon dissections and Schr\"{o}der trees and review some basic definitions and properties of toric varieties. In Section~\ref{sec:construction}, we construct a toric variety of Schr\"{o}der type and show that toric varieties of Schr\"{o}der type are Fano generalized Bott manifolds. In Section~\ref{sec:cohomology}, we give an explicit description of the cohomology ring of a toric variety of Schr\"{o}der type in terms of the associated Schr\"{o}der tree. Section~\ref{sec:proof} is devoted to prove Theorem~\ref{thm1}. In Section~\ref{sec:further}, we give another evidence for the affirmative answer to Problem~\ref{prob1}.

\section{Preliminaries}\label{sec:preliminaries}
In this section, we introduce the Etherington's bijection between the set of polygon dissections and that of Schr\"{o}der trees in \cite{Ether} and review some basic facts on toric varieties.

\subsection{Polygon dissections and Schr\"{o}der trees}

We give the labels $0,1,\dots,n+1$ to the vertices of a regular polygon $\mathsf{P}_{n+2}$ with $n+2$ vertices counterclockwise.
We call the edge $\{0,n+1\}$ the \emph{distinguished edge} of $\mathsf{P}_{n+2}$. 
For simplicity, we also denote by $D$ the set of all polygons in the dissection $D$ of $\mathsf{P}_{n+2}$.
For each polygon $\mathsf{P}^{(i)}$ in $D$, we call the edge connecting the minimal and the maximal vertices the distinguished edge of $\mathsf{P}^{(i)}$ in $D$. Let $E(\mathsf{P}^{(i)})$ denote the set of edges of $\mathsf{P}^{(i)}$ except the distinguished edge of $\mathsf{P}^{(i)}$ in $D$.
Let $E(D)$ be the set of all the edges in $D$ except the distinguished edge of $\mathsf{P}_{n+2}$. Then $E(D)$ is decomposed into the edge sets $E(\mathsf{P}^{(i)})$, $i=1,\dots,k$.

\begin{ex}
Let $D$ be the $4$-dissection of $\mathsf{P}_{10}$ in Figure~\ref{fig:four-dissect}. Then $\{0,9\}$ is the distinguished edge of $\mathsf{P}_{10}$ and it is also the distinguished edge of $\mathsf{P}^{(1)}$ in $D$. The edges $\{0,7\}, \{0,3\}$, and $\{3,7\}$ are the distinguished edges of $\mathsf{P}^{(2)}$, $\mathsf{P}^{(3)}$, and $\mathsf{P}^{(4)}$, respectively. Note that in Figure~\ref{fig:four-dissect}, the distinguished edges are represented by dashed lines.
\end{ex}

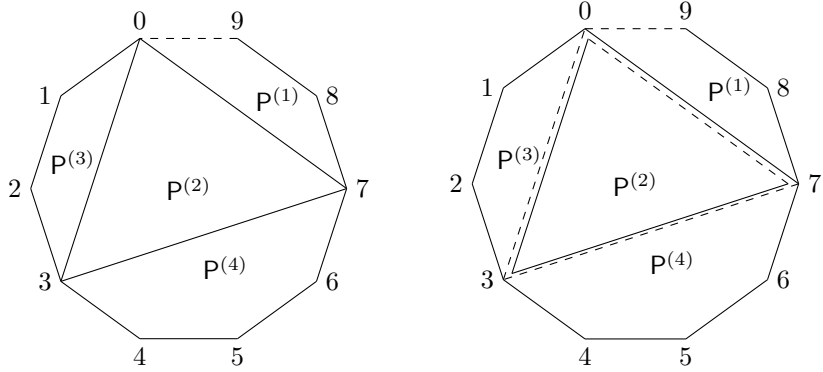
\begin{figure}[hbt]
\begin{center}
\begin{subfigure}[b]{0.4\textwidth}
\centering
	\begin{tikzpicture}[scale=.7]
\foreach \x in {2,...,11}{
	\coordinate (\x)  at (3*36-\x*36:3cm) ;	
} 

\node[above] at (11) {$9$};
\foreach \x/\y in {2/8,3/7,4/6}{
	\node[right] at (\x) {$\y$};
}
\foreach \x/\y in {5/5,6/4}{
	\node[below] at (\x) {$\y$};
}
\foreach \x/\y in {7/3,8/2,9/1}{
	\node[left] at (\x) {$\y$};
}
\node[above] at (10) {$0$};

\draw[dashed] (11) to (10);
\draw
(10) edge (9)
(9) edge (8)
(8) edge (7)
(7) edge (6)
(6) edge (5)
(5) edge (4)
(4) edge (3)
(3) edge (2)
(2) edge (11)
(10) edge (7)
(7) edge (3)
(3) edge (10);

\node at (0,0) {$\mathsf{P}^{(2)}$};
\node at (1.7,1.7) {$\mathsf{P}^{(1)}$};
\node at (-2.2,0.5) {$\mathsf{P}^{(3)}$};
\node at (0.7,-1.5) {$\mathsf{P}^{(4)}$};
\end{tikzpicture}
\end{subfigure}
\qquad
\begin{subfigure}[b]{0.4\textwidth}
\centering
	\begin{tikzpicture}[scale=.7]
\foreach \x in {2,...,11}{
	\coordinate (\x)  at (3*36-\x*36:2.9cm) ;	
}

\coordinate (101) at (3*36-10*36:3.1cm);
\coordinate (91) at (3*36-9*36:3.1cm);
\coordinate (81) at (3*36-8*36:3.1cm);
\coordinate (71) at (3*36-7*36:3.1cm);
\draw
(101) edge (91)
(91) edge (81)
(81) edge (71);
\draw[dashed] (71) to (101);

\draw
(10) edge (7)
(7) edge (3);
\draw[dashed] (3) to (10);

\coordinate (61) at (3*36-6*36:3.1cm);
\coordinate (51) at (3*36-5*36:3.1cm);
\coordinate (41) at (3*36-4*36:3.1cm);
\coordinate (31) at (3*36-3*36:3.1cm);
\draw
(71) edge (61)
(61) edge (51)
(51) edge (41)
(41) edge (31);
\draw[dashed] (31) to (71);

\coordinate (21) at (3*36-2*36:3.1cm);
\coordinate (111) at (3*36-11*36:3.1cm);
\draw
(31) edge (21)
(21) edge (111)
(101) edge (31);
\draw [dashed] (101) to (111);

\node[above] at (111) {$9$};
\node[above] at (101) {$0$};
\node[right] at (21) {$8$};
\node[right] at (31) {$7$};
\node[right] at (41) {$6$};
\node[below] at (51) {$5$};
\node[below] at (61) {$4$};
\node[left] at (71) {$3$};
\node[left] at (81) {$2$};
\node[left] at (91) {$1$};

\node at (0,0) {$\mathsf{P}^{(2)}$};
\node at (1.8,1.8) {$\mathsf{P}^{(1)}$};
\node at (-2.2,0.5) {$\mathsf{P}^{(3)}$};
\node at (0.7,-1.5) {$\mathsf{P}^{(4)}$};

\end{tikzpicture}
\end{subfigure}
\end{center}
\caption{An example of a $4$-dissection of $ \mathsf{P}_{10}$}\label{fig:four-dissect}
\end{figure}

In the following, we show that there is a one-to-one correspondence between the set of $k$-dissection of $\mathsf{P}_{n+2}$ and the set of Schr\"{o}der trees with $k$ internal vertices and $n+1$ leaves.

Let \( D \) be a \( k \)-dissection of  \( \mathsf{P}_{n+2} \). The rooted plane tree \( \tau(D) \) associated with \( D \) is constructed as follows. Let \( e \) denote the distinguished edge of \( \mathsf{P}_{n+2} \). Then \( e \) is contained in a unique polygon in \( D \), say \( \mathsf{Q} \). If \( | E(\mathsf{Q}) |=m \), then by removing \( e \) and the interior of \( \mathsf{Q} \) from \( D \), we have \( m \) dissections \( D_1, D_2, \dots, D_m \), reading counterclockwise from \( e \) along the boundary of \( \mathsf{Q} \), such that \( D_{i} \) and \( D_{i+1} \) intersect at a vertex for \( i=1,2,\dots, m-1 \). Now define recursively \( \tau(D) \) to be the rooted plane tree whose sub-trees of the root are \( \tau(D_1), \tau(D_2), \dots, \tau(D_k) \) in that order. Then the resultant \( \tau(D) \) forms a Schr\"{o}der tree with \( k \) internal vertices and \( n+1 \) leaves as we desired. Note that if $D_i$ is an edge for some $i=1,\dots,k$, then $\tau(D_i)$ is a single vertex graph.

In order to consider the inverse from Schr\"{o}der trees to polygon dissections, we define recursively a labeling of Schr\"{o}der trees. Given a Schr\"{o}der tree \(T\) with \(n+1\) leaves, let
\[
\phi : V(T) \rightarrow \{\{i,j\}\mid 0\leq i <j \leq n+1\},
\] 
where 
\[
\phi(v)=\begin{cases}
\{0,n+1\} & \text{if \(v\) is the root},\\
\{i-1,i\} & \text{if \(v\) is the \(i\)th leaf in the preorder\footnote{It visits from root-left-right.} listing of \(T\)}, and\\
\{i,j\} & \text{if \(v\) is an internal vertex 
whose left-most and right-most}\\
& \text{children are 
labeled by \(\{i,\bullet\}\) and \(\{\bullet,j\}\), respectively.}
\end{cases}
\]

Note that for a dissection \(D\), there is a polygon \(\mathsf{P}^{(i)}\) in \(D\) with vertices labeled by \(v_0< v_1< \cdots < v_{\ell} \) counterclockwise if and only if there is an internal vertex \(v\in V(\tau(D))\), which is labeled by \(\{v_0,v_{\ell}\}\) and has \(\ell\) children labeled by \(\{v_{j-1},v_{j}\}\) for \(j=1,2,\dots, \ell\) from left to right.
Hence, given a Schr\"{o}der tree \(T\), we get the corresponding dissection \(\tau^{-1}(T)\). 
Sometimes, we use the simple notation \(\phi(v)=ij\) instead of \(\phi(v)=\{i,j\}\).
Figure \ref{fig:rooted tree} shows a simple way of finding the correspondence between \(D\) and \( \tau(D) \). Note that in Figure \ref{fig:rooted tree}, the root is represented by \( \odot \). 

\begin{figure}[hbt]
\begin{center}
\begin{subfigure}[b]{0.4\textwidth}
\centering
	\begin{tikzpicture}[scale=.8]
\foreach \x in {2,...,11}{
	\coordinate (\x)  at (3*36-\x*36:3cm) ;	
} 

\node[above] at (11) {$9$};
\foreach \x/\y in {2/8,3/7,4/6}{
	\node[right] at (\x) {$\y$};
}
\foreach \x/\y in {5/5,6/4}{
	\node[below] at (\x) {$\y$};
}
\foreach \x/\y in {7/3,8/2,9/1}{
	\node[left] at (\x) {$\y$};
}
\node[above] at (10) {$0$};

\draw[dashed] (11) to (10);
\draw
(10) edge (9)
(9) edge (8)
(8) edge (7)
(7) edge (6)
(6) edge (5)
(5) edge (4)
(4) edge (3)
(3) edge (2)
(2) edge (11)
(10) edge (7)
(7) edge (3)
(3) edge (10);

\foreach \x in {1,2,3,...,9}{
	\coordinate (\x)  at (-18+3*36-\x*36:3.3cm);
}
\foreach \x in {7,8,9}{
	\draw (-2.2,0.5) edge (\x);
}
\foreach \x in {3,4,5,6}{
	\draw (0.7,-1.9) edge (\x);
}

\draw 
(1.35,1.79) edge (1) 
(1.35,1.79) edge (2) 
(1.35,1.79) edge (0,0)
(0,0) edge (-2.2,0.5)
(0,0) edge  (0.7,-1.9);

\foreach \x in {1,2,3,...,9}{
	\node at (\x) {$\bullet$};
} 

\node at (0,0) {$\bullet$};
\node at (1.35,1.79) {$\odot$};
\node at (-2.2,0.5) {$\bullet$};
\node at (0.7,-1.9) {$\bullet$};

\end{tikzpicture}
\end{subfigure}
\qquad \quad
\begin{subfigure}[b]{0.4\textwidth}
\centering
\begin{tikzpicture}[scale=.8]
\node at (-4,-1.5) {$\leftrightarrow$};
\node at (-4,-1) {$\tau$};

\coordinate (0) at (0,0);
\coordinate (1) at (-1,-1.5);
\coordinate (2) at (0,-1.5);
\coordinate (3) at (1,-1.5);
\coordinate (4) at (-3,-3);
\coordinate (5) at (1,-3);
\coordinate (6) at (-4,-4.5);
\coordinate (7) at (-3,-4.5);
\coordinate (8) at (-2,-4.5);
\coordinate (9) at (-0.5,-4.5);
\coordinate (10) at (0.5,-4.5);
\coordinate (11) at (1.5,-4.5);
\coordinate (12) at (2.5,-4.5);

\node[above] at (0,0.1) {$09$};
\node[below] at (1) {$07$};
\node[below] at (2) {$78$};
\node[below] at (3) {$89$};
\node[above] at (4) {$03$};
\node[above] at (5) {$37$};
\node[below] at (6) {$01$};
\node[below] at (7) {$12$};
\node[below] at (8) {$23$};
\node[below] at (9) {$34$};
\node[below] at (10) {$45$};
\node[below] at (11) {$56$};
\node[below] at (12) {$67$};

\node at (0) {$\odot$};
\node at (0,-5.3) {};
\foreach \x in {1,2,...,12}{
	\node at (\x) {$\bullet$};
}

\foreach \x in {1,2,3}{
	\draw (0) edge (\x);
}

\foreach \x in {4,5}{
	\draw (1) edge (\x);
}

\foreach \x in {6,7,8}{
	\draw (4) edge (\x);
}
\foreach \x in {9,10,11,12}{
	\draw (5) edge (\x);
}
\node at (2.8,0) {};
\end{tikzpicture}
\end{subfigure}
\end{center}
\caption{A dissection of $\mathsf{P}_{10}$ and its corresponding Schr\"{o}der tree}\label{fig:rooted tree}
\end{figure}
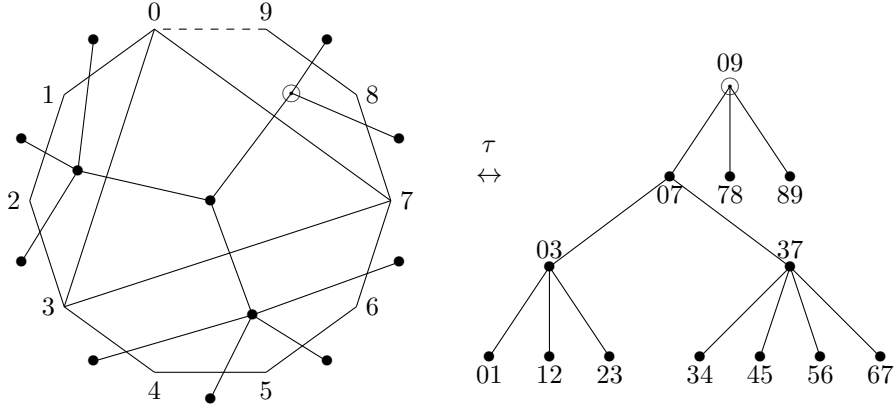

\subsection{Toric varieties}
 A \emph{toric variety} of complex dimension~$n$ is a normal complex algebraic variety with an algebraic action of $(\C^\ast)^n$ having an open dense orbit. A fundamental result of toric geometry is that there is a bijective correspondence between toric varieties of complex dimension~$n$ and rational fans of real dimension~$n$. We can determine geometric or topological properties of toric varieties combinatorially.

Let $\Sigma$ be a fan whose ray generators are $\mathbf{v}_1,\dots,\mathbf{v}_m$ in $\R^n$. The fan $\Sigma$ is \emph{complete} if $\bigcup_{\sigma\in \Sigma}\sigma =\R^n$; $\Sigma$ is \emph{nonsingular} if the $\mathbf{v}_i$'s for the generators of $\sigma$ form a part of a basis of $\Z^n$ for every $\sigma\in\Sigma$; and $\Sigma$ is \emph{polytopal} if it is the normal fan of a lattice polytope. Then  a toric variety $X$ is compact (respectively, smooth and projective) if and only if the associated fan $\Sigma_X$ is complete (respectively, nonsingular and polytopal).

We can also describe the cohomology ring of a smooth compact toric variety explicitly using the associated fan.

\begin{thm}[{\cite{Danilov,Jurkiewicz}}]\label{thm:Da-Ju}
Let $\Sigma$ be a complete nonsingular fan whose ray generators are $\mathbf{v}_1,\dots,\mathbf{v}_m$.
Let $X_\Sigma$ be the toric variety associated with the fan $\Sigma$.
Then the (integral) cohomology ring $H^\ast(X_\Sigma)$ is isomorphic to $\Z[x_1,\dots,x_m]/\mathcal{I}$ as a graded ring, where $\mathcal{I}$ is the ideal generated by the following two types of elements:
\begin{enumerate}
    \item $\prod_{i\in I}x_i$, where $\{\mathbf{v}_i\mid i\in I\}$ does not form a cone in $\Sigma$, and
    \item $\sum_{i=1}^m \langle \mathbf{u},\mathbf{v}_i\rangle x_i$ for any $\mathbf{u}\in \Z^n$,
\end{enumerate}
where $\langle\:,\:\rangle$ denotes the standard scalar product on $\Z^n$.
\end{thm}

A projective smooth variety $X$ is \emph{Fano} if the anticanonical divisor $-K_X$ is ample.
There is a combinatorial way to determine whether a smooth projective toric variety is Fano.
For a projective fan $\Sigma$, a subset $R$ of the primitive ray vectors is called a \emph{primitive collection} of $\Sigma$ if 
\[
\Cone(R) \notin \Sigma
\quad \text{ but }\quad \Cone(R \setminus \{\mathbf{u}\}) \in \Sigma \qquad \text{ for every }\mathbf{u} \in R.
\] Note that if $\Sigma_\mathsf{P}$ is the normal fan of a lattice polytope $P$, then primitive collections of $\Sigma_\mathsf{P}$ correspond to the minimal non-faces of ${P}$.
For a primitive collection $R = \{\mathbf{u}'_1, \dots,\mathbf{u}'_{\ell}\}$, we get $\mathbf{u}'_1 + \cdots+\mathbf{u}'_{\ell}=\boldsymbol{0}$ or there exists a unique cone $\sigma$ such that $\mathbf{u}'_1 + \cdots+\mathbf{u}'_{\ell}$ is in the interior of $\sigma$. That is,
\begin{equation}\label{eq:primitive}
	\mathbf{u}'_1 + \cdots+\mathbf{u}'_{\ell}=\begin{cases}
	\boldsymbol{0}, &\text{ or }\\
	a_1 \mathbf{u}_1 + \cdots+ a_{r} \mathbf{u}_{r},&{}
	\end{cases}
\end{equation}
where $\mathbf{u}_1,\dots,\mathbf{u}_{r}$ are the primitive generators of $\sigma$ and $a_1,\dots,a_{r}$ are positive integers. We call \eqref{eq:primitive} a primitive relation, and the \emph{degree} $\deg R$ of a primitive collection $R$ is defined to be $\ell - (a_1+\cdots+a_r)$.
Batyrev~\cite{Batyrev} gave a criterion for a smooth projective toric variety to be Fano.
\begin{prop}[{\cite[Proposition~2.3.6]{Batyrev}}]\label{prop:batyrev}
A projective toric  variety $X_\Sigma$ is Fano when $\mathrm{deg}(R)>0$ for every primitive collection $R$ of $\Sigma$. 
\end{prop}

We can also classify smooth Fano toric varieties up to isomorphism using the primitive relations.
\begin{prop}[{\cite[Proposition 2.1.8 and Theorem 2.2.4]{Batyrev}}]\label{prop:batyrev-iso}
Two smooth Fano toric varieties $X_\Sigma$ and $X_{\Sigma'}$ are isomorphic as toric varieties if and only if there is a bijection between the sets of rays of $\Sigma$ and $\Sigma'$ inducing a bijection between maximal cones and preserving the primitive relations.
\end{prop}

\section{Smooth toric varieties associated with polygon dissections}\label{sec:construction}

In this section, we construct a toric variety of Schr\"{o}der type and show that they are Fano generalized Bott manifolds. We first construct a nonsingular polytopal fan $\Sigma_D$ associated with a dissection $D$ and then show that the toric variety associated with ${\Sigma_D}$ is a Fano generalized Bott manifold.

We first recall the notion of generalized Bott manifolds.

\begin{defn}[\cite{CMS10_trans}]
	A \emph{generalized Bott tower} $\B_\bullet$  is   an iterated $\C P^{n_{i}}$-bundle:
	\begin{equation}\label{eq_GB}
	\begin{tikzcd}[row sep = 0.2cm]
	\mathcal{B}_{m} \arrow[r, "\pi_{m}"] \arrow[d, equal]&  
	\mathcal{B}_{m-1} \arrow[r, "\pi_{m-1}"] &
	\cdots \arrow[r, "\pi_2"] &
	\mathcal{B}_1 \arrow[r, "\pi_1"] \arrow[d, equal]&
	\mathcal{B}_0, \arrow[d, equal]\\
	P(\underline{\C}\oplus \xi_{m-1}^{1}\oplus \cdots \oplus \xi_{m-1}^{n_{m}}) & & & \C P^{n_{1}} & \{\text{a point}\}
	\end{tikzcd} 
	\end{equation}
	where each $\mathcal{B}_{i}$ is the complex projectivization of the Whitney sum of holomorphic line bundles~$\xi_{i-1}^{k}$ $(1\leq k\leq n_{i})$ and the trivial line 
	bundle $\underline{\C}$ over $\mathcal{B}_{i-1}$.
	We call $\B_m$ an $m$-stage {\em generalized Bott manifold}. 
\end{defn}

When $n_i=1$ for all $i=1,\dots,m$, we omit `generalized.' That is, the tower in~\eqref{eq_GB} is called a \emph{Bott tower} and the total space $\mathcal{B}_m$ is called a \emph{Bott manifold}.

Note that a generalized Bott manifold in~\eqref{eq_GB} is a projective smooth toric variety, and the associated lattice polytope is combinatorially equivalent to $\prod_{i=1}^m \Delta^{n_i}$.
Conversely, if a smooth lattice polytope $P$ is combinatorially equivalent to a product of simplices, then the toric variety associated with the polytope $P$ is a generalized Bott manifold. See~\cite{CMS10_trans, CMS_osaka} for more details.

Recall that we give the labels $0,1,\dots,n+1$ to the vertices of $\mathsf{P}_{n+2}$ counterclockwise. Let us construct the fan $\Sigma_0$ as follows. 
\begin{enumerate}
    \item To each edge $\{i,j\}$ ($i<j$) in $E(\mathsf{P}_{n+2})$, we associate the ray generated by the vector $\mathbf{e}_j-\mathbf{e}_i$, where $\mathbf{e}_0=\mathbf{e}_{n+1}=\boldsymbol{0}$ and $\mathbf{e}_1,\dots,\mathbf{e}_n$ are the standard basis vectors of $\mathbb{R}^n$.
    \item Every proper subset of $E(\mathsf{P}_{n+2})$ constructs the cone generated by the corresponding rays.
\end{enumerate}
Note that the empty set corresponds to the zero-dimensional cone and the toric variety associated with the fan $\Sigma_0$ is the complex projective space $\mathbb{C}P^n$.

Now we associate a dissection of $\mathsf{P}_{n+2}$ to an iterated blowing up of the fan~$\Sigma_0$. We give an order on the edges in the interior of $D$ according to the preorder listing of the associated Schr\"oder tree $\tau(D)$. 
Suppose that the edges in the interior of a $k$-dissection $D$ of  $\mathsf{P}_{n+2}$  are ordered as follows:
\[
\{i_1,j_1\},\{i_2,j_2\},\dots,\{i_{k-1},j_{k-1}\}
\]
according to the preorder listing of $\tau(D)$.
For simplicity, we assume that $i_\ell<j_\ell$ for all $\ell=1,\dots,k-1$.
Then we define the fan $\Sigma_D$ associated with $D$ as follows. We first subdivide the cone of $\Sigma_0$ corresponding to the set $\{\{i_1,i_{1}+1\},\dots,\{j_{1}-1,j_{1}\}\}$ by adding the ray generated by $\mathbf{e}_{j_1}-\mathbf{e}_{i_1}$. Let $\Sigma_1$ be the resulting fan. Now we subdivide the cone of $\Sigma_1$ corresponding to the set generated by $\{\{i_2,i_{2}+1\},\dots,\{j_{2}-1,j_{2}\}\}$ by adding the ray generated by $\mathbf{e}_{j_2}-\mathbf{e}_{i_2}$. Continuing this process until the last edge $\{i_{k-1},j_{k-1}\}$, we get a fan $\Sigma_D$.
Note that the vector $\mathbf{e}_{j_\ell}-\mathbf{e}_{i_\ell}$ is the sum of the primitive ray generators corresponding to $\{i_\ell,i_{\ell}+1\},\dots,\{j_{\ell}-1,j_{\ell}\}$, and $\Sigma_D$ is also smooth.

\begin{defn}
    For a dissection $D$, we say that the toric variety $X_D$ associated with the fan $\Sigma_D$ is \emph{of Schr\"{o}der type}.
\end{defn}

Note that $\Sigma_D$ is also polytopal since $\Sigma_0$ is polytopal. Hence there is a lattice polytope $P_D$ whose normal fan is $\Sigma_D$.
We can get the polytope $P_D$ by truncating faces of the simplex $\Delta^n$. We denote the facets of $\Delta^n$ by $F_{i,i+1}$ for $i=0,1,\dots,n$. Then the polytope $P_D$ is obtained from $\Delta^n$ by truncating the face $F_{i_\ell,i_{\ell+1}}\cap \dots \cap F_{j_{\ell-1},j_{\ell}}$ repeatedly, where $\ell$ starts from $1$ and ends at $k-1$. We denote by $F_{i,\ell}$ the new facet obtained from the truncation of the face $F_{i_\ell,i_{\ell+1}}\cap \dots \cap F_{j_{\ell-1},j_{\ell}}$. See Figure~\ref{fig:poly-dissect}.

\begin{figure}[hbt]
    \centering
    \begin{subfigure}[b]{0.45\textwidth}
    \centering
    \begin{tikzpicture}[scale=.5]
        \foreach \x in {0,...,4}{
	        \coordinate (\x)  at (54+\x*72:2.5cm) ;	
        }
        \node[above] at (0) {$4$};
        \node[right] at (4) {$3$};
        \node[below] at (3) {$2$};
        \node[left] at (2) {$1$};
        \node[above] at (1) {$0$};
        \draw (1)--(2)--(3)--(4)--(0);
        \draw[dashed] (1)--(0);
        \draw (1)--(4);
        
        \coordinate (a) at (4,-0.5);
        \coordinate (b) at (6.5,-1.5);
        \coordinate (c) at (8,-0.5);
        \coordinate (d) at (6,2.5);
        \coordinate (ab) at (4+0.7*2.5,-0.5-0.7*1);
        \coordinate (bc) at (8-0.7*1.5,-0.5-0.7*1);
        \coordinate (bd) at (6+0.7*0.5,2.5-0.7*4);
        \draw[dashed, color=gray!70] (a)--(c);
        \draw (a)--(ab)--(bc)--(c)--(d)--cycle;
        \draw (d)--(bd);
        \filldraw[fill=gray!70] (ab)--(bc)--(bd)--cycle;
        
        \node at (5.5,0.5) {\tiny $F_{01}$};
        \node at (6.9,0.5) {\tiny $F_{12}$};
        \node at (4,-1.5) {\tiny $F_{23}$};
        \node at (8.4,1.5) {\tiny $F_{34}$};
        
        \draw[->] (4.4,-1.5) to [bend right=45] (5.5,-0.9);
        \draw[->] (8,1.5) to [bend right=45] (7,0.7);
    \end{tikzpicture}
    \subcaption{The gray triangle is $F_{03}$.}
    \end{subfigure}
    \quad
    \begin{subfigure}[b]{0.45\textwidth}
    \centering
    \begin{tikzpicture}[scale=.5]
        \foreach \x in {0,...,4}{
	        \coordinate (\x)  at (54+\x*72:2.5cm) ;	
        }
        \node[above] at (0) {$4$};
        \node[right] at (4) {$3$};
        \node[below] at (3) {$2$};
        \node[left] at (2) {$1$};
        \node[above] at (1) {$0$};
        \draw (1)--(2)--(3)--(4)--(0);
        \draw[dashed] (1)--(0);
        \draw (1)--(4);
        \draw (2)--(4);

        \coordinate (a) at (4,-0.5);
        \coordinate (b) at (6.5,-1.5);
        \coordinate (c) at (8,-0.5);
        \coordinate (d) at (6,2.5);
        \coordinate (ab) at (4+0.7*2.5,-0.5-0.7*1);
        \coordinate (bc) at (8-0.7*1.5,-0.5-0.7*1);
        \coordinate (bd) at (6+0.7*0.5,2.5-0.7*4);

        \coordinate (bcd) at (0.7*8-0.49*1.5+0.3*6+0.21*0.5,-0.7*0.5-0.49*1+0.3*2.5-0.21*4);
        \coordinate (abc) at (0.7*8-0.49*1.5+0.3*4+0.21*2.5,-0.7*0.5-0.49*1-0.3*0.5-0.21*1);
        \coordinate (cd) at (0.9*8+0.1*6,-0.9*0.5+0.1*2.5);
        \coordinate (ca) at (0.9*8+0.1*4,-0.9*0.5-0.1*0.5);
        \draw[dashed, color=gray!70] (a)--(ca);
        \draw (a)--(ab)--(abc)--(ca)--(cd)--(d)--cycle;
        \draw (ab)--(abc)--(bcd)--(bd)--cycle;
        \draw (d)--(bd);
        
        \filldraw[fill=gray!70] (bcd)--(abc)--(ca)--(cd)--cycle;
        \filldraw[fill=white] (bd)--(ab)--(abc)--(bcd)--cycle;
        
        \node at (5.5,0.5) {\tiny $F_{01}$};
        \node at (6.9,0.5) {\tiny $F_{12}$};
        \node at (4,-1.5) {\tiny $F_{23}$};
        \node at (8.4,1.5) {\tiny $F_{34}$};
        
        \draw[->] (4.4,-1.5) to [bend right=45] (5.5,-0.9);
        \draw[->] (8,1.5) to [bend right=45] (7,0.7);
    \end{tikzpicture}
    \subcaption{The gray square is $F_{13}$.}
    \end{subfigure}
    \caption{Construction of $P_D$ using face truncations}
    \label{fig:poly-dissect}
\end{figure}
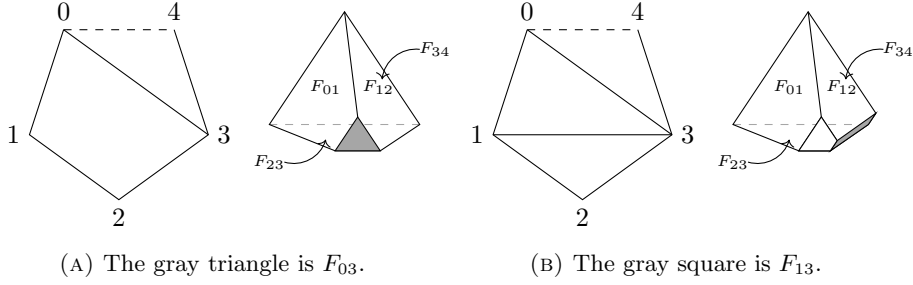

\begin{lem}\label{lem:comb_poly}
    For a dissection $D$ of $\mathsf{P}_{n+2}$, the polytope $P_D$ is combinatorially equivalent to a product of simplices.
\end{lem}
\begin{proof}
    Let $D$ be a $k$-dissection of $\mathsf{P}_{n+2}$ and $E(D)=\cup_{i=1}^k E(\mathsf{P}^{(i)})$. If $k=1$, then every proper subset of $E(\mathsf{P}_{n+2})$ corresponds to a face of $\Delta^n$. When $k=2$, assume
    \begin{align*}
        E(\mathsf{P}^{(1)})&=\{\{0,1\},\dots,\{i_1-1,i_1\},\{i_1,j_1\},\{j_1,j_1+1\},\dots,\{n,n+1\}\},\\
        E(\mathsf{P}^{(2)})&=\{\{i_1,i_1+1\},\dots,\{j_1-1,j_1\}\}.
    \end{align*}
    Then every proper subset of $E(D)$ forms a cone in $\Sigma_D$ unless it is $E(\mathsf{P}^{(1)})$ or $E(\mathsf{P}^{(2)})$.
    Hence every proper subset of $E(D)$ corresponds to a face of $P_D$ except $E(\mathsf{P}^{(1)})$ and $E(\mathsf{P}^{(2)})$.
    Thus $P_D=\Delta^{|E(\mathsf{P}^{(1)})|-1)}\times \Delta^{|E(\mathsf{P}^{(2)}))|-1}$.
    In general, for arbitrary $k$, a subset $S$ of $E(D)$ forms a cone in $\Sigma_D$ if and only if $S$ does not contain $E(\mathsf{P}^{(i)})$ for any $i\in [k]$. Thus the polytope $P_D$ is combinatorially equivalent to the product of simplices $\prod_{i=1}^k \Delta^{|E(\mathsf{P}^{(i)})|-1}$.
\end{proof}

Since every smooth projective toric variety over a product of simplices is a generalized Bott manifold (\cite[Theorem 6.4]{CMS_osaka}), every toric variety of Schr\"{o}der type is a generalized Bott manifold. More preciesely, for a $k$-dissection of the polygon $\mathsf{P}_{n+2}$, the toric variety $X_D$ is a $k$-stage generalized Bott manifold of dimension $n$. If $k=n$, we get an $n$-stage Bott manifold.

Note that the primitive integral generator of the ray corresponding to $\{i,j\}$, $i<j$, is $\mathbf{e}_j-\mathbf{e}_i$ and we denote it by $\mathbf{u}_{ij}$.

\begin{thm}\label{thm:Fano}
    Every toric variety of Schr\"{o}der type is a Fano generalized Bott manifold.
\end{thm}
\begin{proof}
    Let $D$ be a $k$-dissection of $\mathsf{P}_{n+2}$ and $E(D)=\cup_{i=1}^k E(\mathsf{P}^{(i)})$. Lemma~\ref{lem:comb_poly} together with \cite[Theorem 6.4]{CMS_osaka} implies that $X_D$ is a $k$-stage generalized Bott manifold. Hence it is enough to show that $X_D$ is Fano.
    
    From the proof of Lemma~\ref{lem:comb_poly}, the primitive collections of $\Sigma_D$ correspond to the edge sets $E(\mathsf{P}^{(i)})$.
    Let $E(\mathsf{P}^{(i)})=\{\{i_1,i_2\},\dots,\{i_{\ell-1},i_\ell\}\}$, where $i_1<i_2<\dots<i_\ell$. Then the associated primitive relation is
    \begin{equation*}\label{eq:primitive-dissect}
        \begin{split}
        &\mathbf{u}_{i_1i_2}+\cdots+\mathbf{u}_{i_{\ell-1}i_\ell}\\
        &=(\mathbf{e}_{i_2}-\mathbf{e}_{i_1})+\dots+(\mathbf{e}_{i_\ell}-\mathbf{e}_{i_{\ell-1}})\\
        &=\mathbf{e}_{i_\ell}-\mathbf{e}_{i_1}.
        \end{split}
    \end{equation*}
    Note that $\{i_1,i_\ell\}$ is the distinguished edge of $\mathsf{P}^{(i)}$ in $D$. Hence the associated primitive relation is
    \begin{equation}\label{eq:primitive-dissect}
        \mathbf{u}_{i_1i_2}+\cdots+\mathbf{u}_{i_{\ell-1}i_\ell}=\begin{cases}
        \boldsymbol{0}&\text{ if }i_1=0\text{ and }i_\ell=n+1,\\
        \mathbf{u}_{i_1i_\ell} &\text{ otherwise.}
        \end{cases}
    \end{equation}
    Note that $\{i_1,i_\ell\}$ is an edge in $E(\mathsf{P}^{(j)})$ for some $j\neq i$.
    Hence every primitive collection of $\Sigma_D$ has a positive degree. 
    By Proposition~\ref{prop:batyrev}, we conclude that
    the toric variety $X_D$ is a Fano generalized Bott manifold.
\end{proof}

It should be noted that the primitive relations in~\eqref{eq:primitive-dissect} define the Schr\"{o}der tree $\tau(D)$ as follows:
\begin{enumerate}
    \item the vertices $v$ such that
    $\phi(v)\in\{i_1i_2,\dots,i_{\ell-1}i_{\ell}\}$
    have the same parent, and
    \item if $\mathbf{u}_{i_1i_2}+\cdots+\mathbf{u}_{i_{\ell-1}i_\ell}=\boldsymbol{0}$, then the parent is the root of $\tau(D)$; otherwise, the parent is the vertex $v_{i_1i_\ell}$.
\end{enumerate}
By applying Proposition~\ref{prop:batyrev-iso} to the toric varieties $X_D$, we get the following.

\begin{thm}\label{thm:variety-tree}
    The toric varieties $X_D$ and $X_{\widetilde{D}}$ of Schr\"{o}der type are isomorphic as varieties if and only if $\tau(D)$ and $\tau(\widetilde{D})$ are isomorphic as unodered rooted trees.
\end{thm}

Let 
\begin{align*}
    \mathcal{SR}(n,k)&=\{\text{Schr\"{o}der trees with $k$ internal vertices and $n$ leaves}\}/\sim, \\ \mathcal{SR}_n&=\bigcup_{k=0}^{n-1} \mathcal{SR}(n,k)
\end{align*}
where two Schr\"{o}der trees are equivalent if they are the same as unordered rooted trees.
Theorem~\ref{thm:variety-tree} says that $|\mathcal{SR}_{n+1}|$ is the number of isomorphism classes of $n$-dimensional toric varieties of Schr\"{o}der type.

Riordan studied on the numbers $|\mathcal{SR}(n,k)|$ and $|\mathcal{SR}_n|$, and he used the term \emph{series-reduced} in~\cite{Riordan}. That is, each element of $\mathcal{SR}_n$ is called a series-reduced rooted tree.
Let 
\[
s_1(y):=1 \quad \text{and} \quad  s_n(y):=\sum_{k\geq 1} s(n,k) y^k \quad (n\geq 2),
\]
where \(s(n,k)\) denotes the size of $\mathcal{SR}(n,k)$.
A recurrence relation for \(s_n(y)\) is given as follows.

\begin{prop}[{\cite[Equation~(11)]{Riordan}}]
For \(n\geq 2\), 
\[
n(1+y)s_n(y)=ys_n^{*}(y)-s_{n-1}^{*}(y)+(1+y)\sum_{i=1}^{n-1}s_{i}^{*}(y)s_{n-i}(y),
\]
where
\[
s_1^{*}(y):=1 \quad \text{and} \quad s_n^{*}(y):=\sum_{d}d s_d(y^{n/d})
\]
with the sum over all divisors of \(n\).
\end{prop}

The numbers \(s(n,k)\) and \(s_n(1)\) for \(n\leq 10\) are given in  Table~\ref{table}. Clearly, $s(n,1)=1$ and $s(n,2)=n-2$, and it is known that \(s(n,3)=\frac{(n-3)(n-2)}{2}+\lfloor \frac{(n-2)^2}{4}\rfloor\).
However, for general $k$, the explicit formula for $s(n,k)$ is unknown. See sequences A106179 and A000699 in OEIS \cite{OEIS} for more numbers \(s(n,k)\) and \(s_n(1)\), respectively. Note that 
\(s_n(1)\) is the total number of series-reduced rooted trees with $n$ leaves. 

\begin{table}[htb!]
\centering
\begin{tabular}{c|ccccccccccccccccccccc}
\noalign{\smallskip}\noalign{\smallskip}
\(k \backslash n\)&& 1 && 2 && 3 && 4 && 5 && 6 && 7 && 8 && 9 && 10 &\(\cdots\)\\
\hline
1 && && 1 && 1 && 1 && 1 && 1 && 1 && 1 && 1 && 1 &\\
2 && && && 1 && 2 && 3 && 4 && 5 && 6 && 7 && 8 &\\
3 && && && && 2 && 5 && 10 && 16 && 24 && 33 && 44 &\\
4 && && && && && 3 && 12 && 29 && 57 && 99 && 157 &\\
5 && && && && && && 6 && 28 && 84 && 192 && 382 &\\
6 && && && && && && && 11 && 66 && 231 && 615 &\\
7 && && && && && && && && 23 && 157 && 634 &\\
8 && && && && && && && && && 46 && 373 &\\
9 && && && && && && && && && && 98 &\\
\(\vdots\) && && && && && && && && && && &\\
\hline
\(s_n(1)\) && 1 && 1 && 2 && 5 && 12 && 33 && 90 && 261 && 766 && 2312 &\\
\end{tabular}
\caption{The numbers \(s(n,k)\) and \(s_n(1)\) for \(1\leq n \leq 10\)}\label{table}
\end{table}

\begin{cor}\label{cor:number}
    The number of isomorphism classes of $n$-dimensional toric varieties of Schr\"{o}der type is equal to $s_{n+1}(1)$, the number of series-reduced rooted trees with $n+1$ leaves.
\end{cor}

\begin{ex}
 Figure~\ref{fig:iso_trees} shows the equivalent classes of Schr\"{o}der trees with four leaves and the corresponding dissections of \(\mathsf{P}_{5}\). This agrees that \(s(4,1)=1\), \(s(4,2)=2\), and \(s(4,3)=2\) so that \(s_4(y)=y+2y^2+2y^3\) as in Table~\ref{table}. By using this and Theorem~\ref{thm:main}, we will see that the cohomology ring of~\(X_{D}\) for a dissection~\(D\) of~\(\mathsf{P}_{5}\) can be written as one of the following forms:
\begin{align*}
&\mathbb{Z}[x_{1}]/\langle x_1^4 \rangle, 
\quad \mathbb{Z}[x_{1},x_{2}]/\langle x_1^2, x_2(x_1+x_2)^2 \rangle,
\quad \mathbb{Z}[x_{1},x_{2}]/\langle x_1^3, x_2(x_1+x_2) \rangle,\\
&\mathbb{Z}[x_{1},x_{2},x_{3}]/\langle x_1^2,x_2(x_1+x_2),x_3(x_1+x_2+x_3) \rangle,\\
& \mathbb{Z}[x_{1},x_{2},x_{3}]/\langle x_1^2, x_2^2, x_3(-x_1+x_2+x_3) \rangle.
\end{align*}
\end{ex}

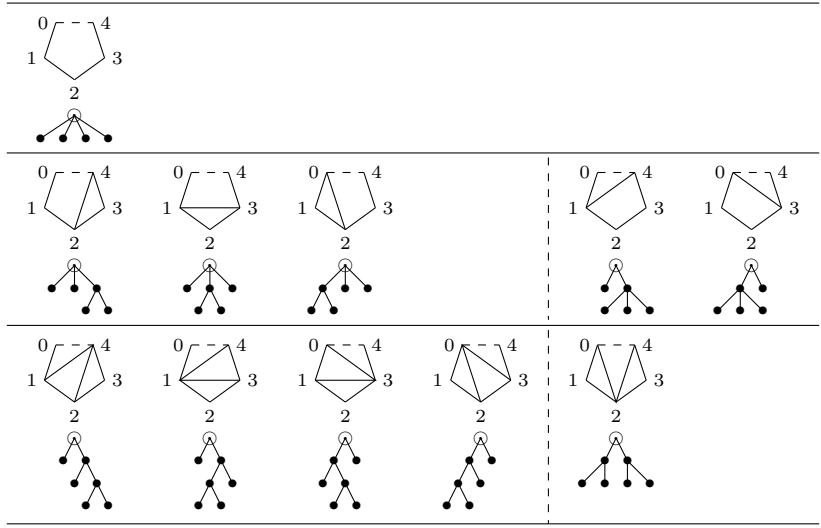
\begin{figure}[hbt]
\scriptsize{

\centering
\begin{tikzpicture}[scale=.6]
\draw (-1.5,1) -- (16.5,1);
\node at (0,-1.5) {\(\odot\)};
\foreach \x in {1,...,4}{
	\coordinate (\x)  at (-1.25+\x*0.5,-2);
	\node at (\x) {\(\bullet\)};
	\draw (0,-1.5) -- (\x);
}


\foreach \x in {0,1,...,4}{
	\coordinate (\x)  at (-90+3*72+\x*72:0.7cm) ;	
} 

\foreach \x in {0,1}{
	\node[left] at (\x) {$\x$};
}
\foreach \x in {3,4}{
	\node[right] at (\x) {$\x$};
}
\foreach \x in {2}{
	\node[below] at (\x) {$\x$};
}
\draw[dashed] (0) -- (4);

\foreach \x/\y in {0/1,1/2,2/3,3/4}{
	\draw (\x) -- (\y);
}

\end{tikzpicture}


\begin{tikzpicture}[scale=.6]

\draw (-1.5,1) -- (16.5,1);
\draw[dashed] (10.5,0.9) -- (10.5,-2.7);
\node at (0,-1.5) {\(\odot\)};
\foreach \x in {1,2,3}{
	\coordinate (\x)  at (-1+\x*0.5,-2);
	\node at (\x) {\(\bullet\)};
	\draw (0,-1.5) -- (\x);
}

\foreach \x in {4,5}{
	\coordinate (\x)  at (-1.75+\x*0.5,-2.5);
	\node at (\x) {\(\bullet\)};
	\draw (3) -- (\x);
}

\foreach \x in {0,1,...,4}{
	\coordinate (\x)  at (-90+3*72+\x*72:0.7cm) ;	
} 

\foreach \x in {0,1}{
	\node[left] at (\x) {$\x$};
}
\foreach \x in {3,4}{
	\node[right] at (\x) {$\x$};
}
\foreach \x in {2}{
	\node[below] at (\x) {$\x$};
}
\draw[dashed] (0) -- (4);

\foreach \x/\y in {0/1,1/2,2/3,3/4}{
	\draw (\x) -- (\y);
}
\draw (2) -- (4);


\node at (3,-1.5) {\(\odot\)};
\foreach \x in {1,2,3}{
	\coordinate (\x)  at (2+\x*0.5,-2);
	\node at (\x) {\(\bullet\)};
	\draw (3,-1.5) -- (\x);
}

\foreach \x in {4,5}{
	\coordinate (\x)  at (3-2.25+\x*0.5,-2.5);
	\node at (\x) {\(\bullet\)};
	\draw (2) -- (\x);
}


\foreach \x in {0,1,...,4}{
	\coordinate (\x)  at  ({3+0.7*cos(-90+3*72+\x*72)},{0.7*sin(-90+3*72+\x*72)});
} 

\foreach \x in {0,1}{
	\node[left] at (\x) {$\x$};
}
\foreach \x in {3,4}{
	\node[right] at (\x) {$\x$};
}
\foreach \x in {2}{
	\node[below] at (\x) {$\x$};
}
\draw[dashed] (0) -- (4);

\foreach \x/\y in {0/1,1/2,2/3,3/4}{
	\draw (\x) -- (\y);
}
\draw (1) -- (3);


\node at (6,-1.5) {\(\odot\)};
\foreach \x in {1,2,3}{
	\coordinate (\x)  at (5+\x*0.5,-2);
	\node at (\x) {\(\bullet\)};
	\draw (6,-1.5) -- (\x);
}

\foreach \x in {4,5}{
	\coordinate (\x)  at (6-2.75+\x*0.5,-2.5);
	\node at (\x) {\(\bullet\)};
	\draw (1) -- (\x);
}


\foreach \x in {0,1,...,4}{
	\coordinate (\x)  at ({6+0.7*cos(-90+3*72+\x*72)},{0.7*sin(-90+3*72+\x*72)});
} 

\foreach \x in {0,1}{
	\node[left] at (\x) {$\x$};
}
\foreach \x in {3,4}{
	\node[right] at (\x) {$\x$};
}
\foreach \x in {2}{
	\node[below] at (\x) {$\x$};
}
\draw[dashed] (0) -- (4);

\foreach \x/\y in {0/1,1/2,2/3,3/4}{
	\draw (\x) -- (\y);
}
\draw (2) -- (0);


\node at (12,-1.5) {\(\odot\)};
\foreach \x in {1,2}{
	\coordinate (\x)  at (12-0.75+\x*0.5,-2);
	\node at (\x) {\(\bullet\)};
	\draw (12,-1.5) -- (\x);
}

\foreach \x in {3,4,5}{
	\coordinate (\x)  at (12-1.75+\x*0.5,-2.5);
	\node at (\x) {\(\bullet\)};
	\draw (2) -- (\x);
}


\foreach \x in {0,1,...,4}{
	\coordinate (\x)  at ({12+0.7*cos(-90+3*72+\x*72)},{0.7*sin(-90+3*72+\x*72)});
} 

\foreach \x in {0,1}{
	\node[left] at (\x) {$\x$};
}
\foreach \x in {3,4}{
	\node[right] at (\x) {$\x$};
}
\foreach \x in {2}{
	\node[below] at (\x) {$\x$};
}
\draw[dashed] (0) -- (4);

\foreach \x/\y in {0/1,1/2,2/3,3/4}{
	\draw (\x) -- (\y);
}
\draw (1) -- (4);


\node at (15,-1.5) {\(\odot\)};

\foreach \x in {1,2}{
	\coordinate (\x)  at (15-0.75+\x*0.5,-2);
	\node at (\x) {\(\bullet\)};
	\draw (15,-1.5) -- (\x);
}

\foreach \x in {3,4,5}{
	\coordinate (\x)  at (15-2.25+\x*0.5,-2.5);
	\node at (\x) {\(\bullet\)};
	\draw (1) -- (\x);
}


\foreach \x in {0,1,...,4}{
	\coordinate (\x)  at ({15+0.7*cos(-90+3*72+\x*72)},{0.7*sin(-90+3*72+\x*72)});
} 

\foreach \x in {0,1}{
	\node[left] at (\x) {$\x$};
}
\foreach \x in {3,4}{
	\node[right] at (\x) {$\x$};
}
\foreach \x in {2}{
	\node[below] at (\x) {$\x$};
}
\draw[dashed] (0) -- (4);

\foreach \x/\y in {0/1,1/2,2/3,3/4}{
	\draw (\x) -- (\y);
}
\draw (0) -- (3);
\end{tikzpicture}


\begin{tikzpicture}[scale=.6]

\draw (-1.5,1) -- (16.5,1);
\draw[dashed] (10.5,0.9) -- (10.5,-3.4);
\draw (-1.5,-3.4) -- (16.5,-3.4);


\node at (0,-1.5) {\(\odot\)};
\foreach \x in {1,2}{
	\coordinate (\x)  at (-0.75+\x*0.5,-2);
	\node at (\x) {\(\bullet\)};
	\draw (0,-1.5) -- (\x);
}

\foreach \x in {3,4}{
	\coordinate (\x)  at (-1.5+\x*0.5,-2.5);
	\node at (\x) {\(\bullet\)};
	\draw (2) -- (\x);
}

\foreach \x in {5,6}{
	\coordinate (\x)  at (-2.25+\x*0.5,-3);
	\node at (\x) {\(\bullet\)};
	\draw (4) -- (\x);
}


\foreach \x in {0,1,...,4}{
	\coordinate (\x)  at ({0.7*cos(-90+3*72+\x*72)},{0.7*sin(-90+3*72+\x*72)});
} 

\foreach \x in {0,1}{
	\node[left] at (\x) {$\x$};
}
\foreach \x in {3,4}{
	\node[right] at (\x) {$\x$};
}
\foreach \x in {2}{
	\node[below] at (\x) {$\x$};
}
\draw[dashed] (0) -- (4);

\foreach \x/\y in {0/1,1/2,2/3,3/4}{
	\draw (\x) -- (\y);
}
\draw (1) -- (4) -- (2);


\node at (3,-1.5) {\(\odot\)};

\foreach \x in {1,2}{
	\coordinate (\x)  at (3-0.75+\x*0.5,-2);
	\node at (\x) {\(\bullet\)};
	\draw (3,-1.5) -- (\x);
}

\foreach \x in {3,4}{
	\coordinate (\x)  at (3-1.5+\x*0.5,-2.5);
	\node at (\x) {\(\bullet\)};
	\draw (2) -- (\x);
}

\foreach \x in {5,6}{
	\coordinate (\x)  at (3-2.75+\x*0.5,-3);
	\node at (\x) {\(\bullet\)};
	\draw (3) -- (\x);
}


\foreach \x in {0,1,...,4}{
	\coordinate (\x) at ({3+0.7*cos(-90+3*72+\x*72)},{0.7*sin(-90+3*72+\x*72)});
} 

\foreach \x in {0,1}{
	\node[left] at (\x) {$\x$};
}
\foreach \x in {3,4}{
	\node[right] at (\x) {$\x$};
}
\foreach \x in {2}{
	\node[below] at (\x) {$\x$};
}
\draw[dashed] (0) -- (4);

\foreach \x/\y in {0/1,1/2,2/3,3/4}{
	\draw (\x) -- (\y);
}
\draw (4) -- (1) -- (3);


\node at (6,-1.5) {\(\odot\)};

\foreach \x in {1,2}{
	\coordinate (\x)  at (6-0.75+\x*0.5,-2);
	\node at (\x) {\(\bullet\)};
	\draw (6,-1.5) -- (\x);
}

\foreach \x in {3,4}{
	\coordinate (\x)  at (6-2+\x*0.5,-2.5);
	\node at (\x) {\(\bullet\)};
	\draw (1) -- (\x);
}

\foreach \x in {5,6}{
	\coordinate (\x)  at (6-2.75+\x*0.5,-3);
	\node at (\x) {\(\bullet\)};
	\draw (4) -- (\x);
}


\foreach \x in {0,1,...,4}{
	\coordinate (\x)  at ({6+0.7*cos(-90+3*72+\x*72)},{0.7*sin(-90+3*72+\x*72)});
} 

\foreach \x in {0,1}{
	\node[left] at (\x) {$\x$};
}
\foreach \x in {3,4}{
	\node[right] at (\x) {$\x$};
}
\foreach \x in {2}{
	\node[below] at (\x) {$\x$};
}
\draw[dashed] (0) -- (4);

\foreach \x/\y in {0/1,1/2,2/3,3/4}{
	\draw (\x) -- (\y);
}
\draw (0) -- (3) -- (1);


\node at (9,-1.5) {\(\odot\)};

\foreach \x in {1,2}{
	\coordinate (\x)  at (9-0.75+\x*0.5,-2);
	\node at (\x) {\(\bullet\)};
	\draw (9,-1.5) -- (\x);
}

\foreach \x in {3,4}{
	\coordinate (\x)  at (9-2+\x*0.5,-2.5);
	\node at (\x) {\(\bullet\)};
	\draw (1) -- (\x);
}

\foreach \x in {5,6}{
	\coordinate (\x)  at (9-3.25+\x*0.5,-3);
	\node at (\x) {\(\bullet\)};
	\draw (3) -- (\x);
}


\foreach \x in {0,1,...,4}{
	\coordinate (\x)  at ({9+0.7*cos(-90+3*72+\x*72)},{0.7*sin(-90+3*72+\x*72)});	
} 

\foreach \x in {0,1}{
	\node[left] at (\x) {$\x$};
}
\foreach \x in {3,4}{
	\node[right] at (\x) {$\x$};
}
\foreach \x in {2}{
	\node[below] at (\x) {$\x$};
}
\draw[dashed] (0) -- (4);

\foreach \x/\y in {0/1,1/2,2/3,3/4}{
	\draw (\x) -- (\y);
}
\draw (2) -- (0) -- (3);

\node at (12,-1.5) {\(\odot\)};

\foreach \x in {1,2}{
	\coordinate (\x)  at (12-0.75+\x*0.5,-2);
	\node at (\x) {\(\bullet\)};
	\draw (12,-1.5) -- (\x);
}

\foreach \x in {3,4}{
	\coordinate (\x)  at (12-2.25+\x*0.5,-2.5);
	\node at (\x) {\(\bullet\)};
	\draw (1) -- (\x);
}

\foreach \x in {5,6}{
	\coordinate (\x)  at (12-2.25+\x*0.5,-2.5);
	\node at (\x) {\(\bullet\)};
	\draw (2) -- (\x);
}


\foreach \x in {0,1,...,4}{
	\coordinate (\x)  at ({12+0.7*cos(-90+3*72+\x*72)},{0.7*sin(-90+3*72+\x*72)});	
} 

\foreach \x in {0,1}{
	\node[left] at (\x) {$\x$};
}
\foreach \x in {3,4}{
	\node[right] at (\x) {$\x$};
}
\foreach \x in {2}{
	\node[below] at (\x) {$\x$};
}
\draw[dashed] (0) -- (4);

\foreach \x/\y in {0/1,1/2,2/3,3/4}{
	\draw (\x) -- (\y);
}
\draw (0) -- (2) -- (4);
\end{tikzpicture}

}
\caption{Equivalent classes of Schr\"{o}der trees with four leaves and the corresponding dissections of \(\mathsf{P}_{5}\) }\label{fig:iso_trees}
\end{figure}

\begin{rem}
If a dissection $D$ is a triangulation, then $X_D$ becomes a toric variety of Catalan type in~\cite{LMP_Catalan} and the associated Schr\"{o}der tree $\tau(D)$ is a full binary tree.  The number of isomorphism classes of toric varieties of Catalan type is $s(n+1,n)$, which is known as the Wedderburn-Etherington number $b_{n+1}$.
\end{rem}

Note that not every Fano generalized Bott manifold is associated with a polygon dissection. For instance, the total space of the projective bundle $P(\underline{\C}\oplus\gamma^{\otimes 2})$ over $\C P^3$ is a Fano generalized Bott manifold, but there is no polygon dissection associated with it, where $\gamma$ is the tautological line bundle over $\C P^3$.

\section{Cohomology ring of the toric variety from a polygon dissection}\label{sec:cohomology}

In this section, we describe the cohomology ring of $X_D$ using the Schr\"{o}der tree $\tau(D)$.

Let $v$ be a vertex of $\tau(D)$ with $\phi(v)=\{a,b\}$. Then we denote by $S(v)$ the set of descendants of $v$ whose label is $\{\bullet,b\}$. 
For instance, if $v$ is the vertex labeled by $\{0,7\}$ in Figure~\ref{fig:rooted tree}, then $S(v)=\{u\mid \phi(u)\text{ is }\{3,7\}\text{ or }\{6,7\}\}$. Note that $S(v)$ is empty for a leaf $v$ of $\tau(D)$.

\begin{thm}\label{thm:main}
Given a \(k\)-dissection \(D\) of a polygon, consider the corresponding Schr\"{o}der tree \(\tau(D)\). For \(1\leq i \leq k\), let \(v_i\) be the \(i\)th internal vertex in the preorder listing of \(\tau(D)\). For each \(i\), suppose that \(v_i\) has \(\ell_i\) children \(w_{i1},w_{i2},\dots,w_{i\ell_i}\) from left to right, and \(\phi(w_{i\ell_i})=\{a_i,b_i\}\) with \(a_i<b_i\).  
Then the cohomology ring of \(X_D\) can be written as 
\[
H^*(X_{D})=\mathbb{Z}[x_{a_1b_1},x_{a_2b_2},\dots,x_{a_kb_k}]/\langle p_1,\dots,p_k \rangle,
\]
where 
\[
p_i:=x_{a_ib_i}\prod_{j=1}^{\ell_i-1}
\left(-\sum_{u \in S(w_{ij})}x_{\phi(u)}+\sum_{u \in S(v_i)}x_{\phi(u)}\right).
\]

\end{thm}

\begin{proof}
    Let $D$ be a $k$-dissection of $\mathsf{P}_{n+2}$ and $E(D)=\cup_{i=1}^k E(\mathsf{P}^{(i)})$. Let $E(\mathsf{P}^{(i)})=\{i_1i_2,i_2i_3,\dots,i_{\ell-1}i_\ell\}$, where $i_1<i_2<\dots<i_\ell$. Then there are $k$ generators of first type of $\mathcal{I}$ in Theorem~\ref{thm:Da-Ju}:
    \begin{equation*}
        x_{i_1i_2}\dots x_{i_{\ell-1}i_\ell} x_{i_1i_\ell},
    \end{equation*}
    corresponding to $E(\mathsf{P}^{(i)})$ for $i=1,\dots,k$.
    For each $i\in [n]$, we have the linear relation $q_i=0$, where
    \begin{equation*}
        q_i:=\sum_{\{j,i\}\in E(\mathsf{P}_{n+2})\atop j<i}x_{ji} - \sum_{\{i,j\}\in E(\mathsf{P}_{n+2})\atop i<j}x_{ij},
    \end{equation*}
    which are generators of the second type of $\mathcal{I}$ in Theorem~\ref{thm:Da-Ju}.
    For simplicity, for $e\in E(\mathsf{P}_{n+2})$, let $v_e$ be the vertex in $\tau(D)$ satisfying $\phi(v_e)=e$ and let $r_e$ be the parent of $v_e$. We need to show that for $e\in E(\mathsf{P}_{n+2})$,  if $v_e$ is not the right-most child of $r_e$, then
    \begin{equation}\label{eq:linear}
        x_{e} = -\sum_{u \in S(v_e)}x_{\phi(u)}+\sum_{u \in S(r_e)}x_{\phi(u)},
    \end{equation} 
    and $\phi(u)\in\{a_1b_1,a_2b_2,\dots,a_kb_k\}$ for all $u\in S(v_e) \cup S(r_e)$.
    Suppose that $\phi(v_e)=bc$ and $\phi(r_e)=ad$. By the construction of labeling, it follows that $a\leq b<c \leq d$. In fact, $c<d$ since $v_e$ is not the right-most child of $r_e$. Consider the linear relation $q_c+q_{c+1}+\cdots+q_{d-1}=0$. Since $\{a,a+1,\dots, d\}$ and $\{b,b+1,\dots, c\}$ form sub-dissections,
    \[
    q_c+q_{c+1}+\cdots+q_{d-1}=\sum_{\{i,c\}\in E(\mathsf{P}_{n+2}) \atop b\leq i<c} x_{ic}-\sum_{\{j,d\}\in E(\mathsf{P}_{n+2}) \atop c\leq j<d} x_{jd}
    \]
    so that we have 
    \[
    x_{e}=-\sum_{\{i,c\}\in E(\mathsf{P}_{n+2}) \atop b<i<c} x_{ic}+\sum_{\{j,d\}\in E(\mathsf{P}_{n+2}) \atop c\leq j<d} x_{jd}.
    \]
    Note that if $v$ and $u$ are vertices such that $\phi(v)=i_1j$ and $\phi(u)=i_2j$ with $i_1<i_2<j$, then $u$ is a descendant of $v$, and $u$ is the right-most child of its parent. 
    Thus, for $e\in E(\mathsf{P}_{n+2})$, unless $v_e$ is the right-most child of $r_e$, the element $x_{e}$  can be written as \eqref{eq:linear}, and it is a linear combination of $x_{a_1b_1},\dots, x_{a_kb_k}$. 
\end{proof}

\begin{ex}
Let $D$ be the $4$-dissection of $\mathsf{P}_{10}$ in Figure~\ref{fig:four-dissect}. Then \(\tau(D)\) with the label \(\phi(v)\) for \(v\in V(\tau(D))\) is given in Figure~\ref{fig:rooted tree}. Let us compute the cohomology ring $H^\ast(X_D)$ using Theorem~\ref{thm:main}. Then
\begin{equation*}
    \begin{split}
        &x_{01}=x_{12}=x_{23},\\
        &x_{34}=x_{45}=x_{56}=x_{67},\\
        &x_{03}=-x_{23}+x_{37}+x_{67},\\
        &x_{07}=-x_{37}-x_{67}+x_{89},\\
        &x_{78}=x_{89}.
    \end{split}
\end{equation*}
Hence we get
\begin{equation*}
    \begin{split}
        x_{01}x_{12}x_{23}&=x_{23}^3,\\
        x_{34}x_{45}x_{56}x_{67}&=x_{67}^4,\\
        x_{03}x_{37}&=x_{37}(-x_{23}+x_{37}+x_{67}),\\
        x_{07}x_{78}x_{89}&=x_{89}^2(-x_{37}-x_{67}+x_{89}).
    \end{split}
\end{equation*}
Therefore, the cohomology ring of $X_D$ is
\[
H^*(X_{D})=\mathbb{Z}[x_{23},x_{37},x_{67},x_{89}]/\langle x_{23}^{3},x_{37}(-x_{23}+x_{37}+x_{67}), x_{67}^4, x_{89}^2(-x_{37}-x_{67}+x_{89}) \rangle.
\]
\end{ex}

The following is a direct consequence of Theorem~\ref{thm:main}.

\begin{cor}\label{cor:leaf}
    Let $D$ be a dissection of a polygon and let $\tau(D)$ be its associated Schr\"{o}der tree. For a vertex $w$ with $\phi(w)=\{a_i,b_i\}$ in Theorem~\ref{thm:main}, if the vertex $w$ is a leaf of $\tau(D)$ and $v$ is the parent of $w$, then $x_{a_ib_i}^{\ell_i}=0$ in $H^\ast(X_D)$, where $\ell_i$ is the number of children of $v$.
\end{cor}

Note that the combinatorial type of the associated lattice polytope with a generalized Bott manifold is determined by the cohomology ring. That is, for two generalized Bott manifolds $B$ and $B'$, if the lattice polytope associated with $B$ is combinatorially equivalent to $\prod_{i=1}^m \Delta^{n_i}$ and $H^\ast(B)\cong H^\ast(B')$, then the lattice polytope associated with $B'$ is also combinatorially equivalent to $\prod_{i=1}^m \Delta^{n_i}$. See \cite[Theorem 5.3]{Choi-Panov-Suh} for more details. 
Therefore, for two dissections $D$ and $\widetilde{D}$, if $H^\ast(X_D)\cong H^\ast(X_{\widetilde{D}})$, then $P_D$ and $P_{\widetilde{D}}$ are combinatorially equivalent, so we get the following.

\begin{cor}\label{cor:cohomology-size}
    Let $D$ and $\widetilde{D}$ be $k$-dissection and $\tilde{k}$-dissection, respectively, if $H^\ast(X_D)\cong H^\ast(X_{\widetilde{D}})$, then $k=\tilde{k}$ and the sets $\{|E(\mathsf{P}^{(i)})|\mid i=1,\dots,k\}$ and $\{|E(\widetilde{\mathsf{P}}^{(i)})|\mid i=1,\dots,\tilde{k}\}$ are the same as multi-sets.
\end{cor}

\section{Proof of Theorem~\ref{thm1}}\label{sec:proof}

In this section, we prove Theorem~\ref{thm1}. When $k=n$, it is proved in~\cite{LMP_Catalan}. If $k\leq 2$, then it is clear from Theorem~\ref{thm:Fano} and Corollary~\ref{cor:cohomology-size}. We prove when $k=3$ by dividing three cases. We first divide the three-dissections into two types and show that a cohomology ring determines the type of a three-dissection. After that, we show that for each type, toric varieties associated with three-dissections are classified by cohomology rings up to variety isomorphism.

Let $D$ be a three-dissection of a polygon and let \(m_1,m_2\), and \(m_3\) be the out-degrees of the internal vertices of the Schr\"{o}der tree \(\tau(D)\). It suffices to consider the two types of Sch\"{o}der trees with three internal vertices as in Figure~\ref{fig:Ds0}. Note that  \(m_i\geq 2\) for all \(i=1,2,3\) since \(\tau(D)\) is a Schr\"{o}der tree.
By abuse of notation, we use the same letter $D$ for the Schr\"{o}der tree $\tau(D)$ and we say that \(D\) is of the first (respectively, second) type if its corresponding Schr\"{o}der path is of the first (respectively, second) type.

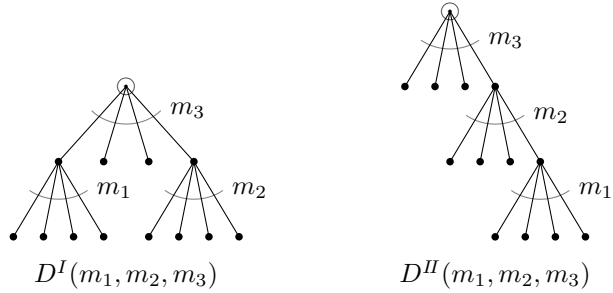
\begin{figure}[ht!]
\centering
\begin{tikzpicture}[scale=1]

\node at (0,0) {$\odot$};

\draw[color=gray,domain=-140:-40] plot ({0.6*cos(\x)}, {0.5*sin(\x)});
\draw[color=gray,domain=-130:-50] plot ({-0.9+0.6*cos(\x)}, {-1+0.5*sin(\x)});
\draw[color=gray,domain=-130:-50] plot ({0.9+0.6*cos(\x)}, {-1+0.5*sin(\x)});

\foreach \i in {-0.9, -0.3, 0.3, 0.9}{
\draw 
(0,0)--(\i,-1);
\fill (\i,-1) circle (0.5mm);
}

\foreach \i in {-0.6, -0.2, 0.2, 0.6}{
\draw 
(-0.9,-1)--(-0.9+\i,-2)
(0.9,-1)--(0.9+\i,-2);
\fill 
(-0.9+\i,-2) circle (0.5mm)
(0.9+\i,-2) circle (0.5mm);
}

\node at (0,-2.5) {\(D^I(m_1,m_2,m_3)\)};
\node[right] at ({0.6*cos(-40)}, {0.5*sin(-40)}) {\(m_3\)};
\node[right] at ({-0.9+0.6*cos(-50)}, {-1+0.5*sin(-50)}) {\(m_1\)};
\node[right] at ({0.9+0.6*cos(-50)}, {-1+0.5*sin(-50)}) {\(m_2\)};
\end{tikzpicture}
\qquad\qquad
\begin{tikzpicture}[scale=1]

\node at (0,0) {$\odot$};

\draw[color=gray,domain=-130:-50] plot ({0.6*cos(\x)}, {0.5*sin(\x)});
\draw[color=gray,domain=-130:-50] plot ({0.6+0.6*cos(\x)}, {-1+0.5*sin(\x)});
\draw[color=gray,domain=-130:-50] plot ({1.2+0.6*cos(\x)}, {-2+0.5*sin(\x)});

\foreach \i in {-0.6, -0.2, 0.2, 0.6}{
\draw 
(0,0)--(\i,-1)
(0.6,-1)--(0.6+\i,-2)
(1.2,-2)--(1.2+\i,-3);
\fill (\i,-1) circle (0.5mm);
\fill (0.6+\i,-2) circle (0.5mm);
\fill (1.2+\i,-3) circle (0.5mm);
}

\node at (0.6,-3.5) {\(D^{I\!I}(m_1,m_2,m_3)\)};
\node[right] at ({0.6*cos(-50)}, {0.5*sin(-50)}) {\(m_3\)};
\node[right] at ({0.6+0.6*cos(-50)}, {-1+0.5*sin(-50)}) {\(m_2\)};
\node[right] at ({1.2+0.6*cos(-50)}, {-2+0.5*sin(-50)}) {\(m_1\)};
\end{tikzpicture}

\caption{Two types of Schr\"oder trees \(D^I\) and \(D^{I\!I}\).}\label{fig:Ds0}
\end{figure}

For a three dissection of each type, we let
\begin{align*}
\mathcal{I}_{D^I}&=\langle x_1^{m_1}, x_2^{m_2}, x_3(-x_1+x_2+x_3)(x_2+x_3)^{m_3-2} \rangle,\\    
\mathcal{I}_{D^{I\!I}}&=\langle x_1^{m_1}, x_2(x_1+x_2)^{m_2-1}, x_3(x_1+x_2+x_3)^{m_3-1} \rangle.
\end{align*} Then from Theorem~\ref{thm:main}, we have
\begin{equation*}
    H^\ast(X_{D^{I}})=\mathbb{Z}[x_1,x_2,x_3]/\mathcal{I}_{D^I},\quad\text{and }\quad  
H^*(X_{D^{I\!I}})=\mathbb{Z}[x_1,x_2,x_3]/\mathcal{I}_{D^{I\!I}}.
\end{equation*}

Let $D$ and $\widetilde{D}$ be three-dissections of a polygon. 
For \(i=1,2,3\), let \(m_i\) (respectively, \( \widetilde{m}_i\)) denote the out-degrees of the internal vertices in \(D\) (respectively, \(\widetilde{D}\)). For \(m=:\min\{m_1,m_2,m_3\}\), let \(n(D)\) denote the number of independent linear terms \(\alpha\) in \(H^*(X_{D})\) such that \(\alpha^m=0\).
Note that if \(H^*(X_D)\cong H^*(X_{\widetilde{D}})\), then \(n(D)=n(\widetilde{D})\), and the multi-sets \( \{m_1,m_2,m_3\}\) and \( \{\widetilde{m}_1,\widetilde{m}_2, \widetilde{m}_3\} \) are the same by Corollary~\ref{cor:cohomology-size}.

We first prove that if $H^\ast(X_D)\cong H^\ast(X_{\widetilde{D}})$, then $D$ and $\widetilde{D}$ are of the same type.
To prove this, we prepare three lemmas about the cohomology of the toric varieties associated with the second type three-dissections.

\begin{lem}\label{lem:type2-coho-1}
    Let \(D=D^{I\!I}(m_1,m_2,m_3)\) and let $m:=\min\{{m}_2,{m}_3\}$. If $m<{m_1}$, then there is no nonzero element in $H^2(X_D)$ whose $m$th power vanishes in $H^\ast(X_D)$.
\end{lem}
\begin{proof}
    Suppose that there exists a nonzero element $px_1+qx_2+rx_3$ in $H^2(X_{{D}})$ such that $(px_1+qx_2+rx_3)^{m}=0$ in $H^\ast(X_{{D}})$. Since $m<{m}_1$, we have
    \begin{equation}\label{eq:lem1}
        (px_1+qx_2+rx_3)^{m}=\alpha x_2(x_1+x_2)^{{m}_2-1} + \beta x_3(x_1+x_2+x_3)^{{m}_3-1}
    \end{equation}
    as polynomials. In any case of $m={m}_2$ and $m={m}_3$, comparing the coefficients of $x_1^{m}$ in both sides of \eqref{eq:lem1}, we get $p=0$. Since \(m_2,m_3\geq 2\) it follows that $\alpha=\beta=0$, which is a contradiction to $px_1+qx_2+rx_3\neq 0$.
\end{proof}

\begin{lem}\label{lem:type2-coho-2}
    Let \(D=D^{I\!I}(m_1,m_2,m_3)\). If $m_1=2$, then $\{(m_2-1)x_1+m_2x_2\}^{m_2}=0$ in $H^\ast(X_D)$.
\end{lem}
\begin{proof}
    Since the coefficient of $x_1x_2^{m_2-1}$ in $\{(m_2-1)x_1+m_2x_2\}^{m_2}$ is 
    $$ m_2(m_2-1)m_2^{m_2-1}=m_2^{m_2}(m_2-1),$$ 
    there exists a polynomial $\alpha$ such that $$\{(m_2-1)x_1+m_2x_2\}^{m_2}=\alpha x_1^2+m_2^{m_2}x_2(x_1+x_2)^{m_2-1}.$$
    This proves the lemma.
\end{proof}

\begin{lem}\label{lem:type2-coho-3}
    Let \(D=D^{I\!I}(m_1,m_2,m_3)\). If $m_1>2$, then there is no element $px_1+qx_2+rx_3 \in H^2(X_D)$ with $(q,r)\neq (0,0)$ such that  $(px_1+qx_2+rx_3)^{m_2}$ or $(px_1+qx_2+rx_3)^{m_3}$ vanishes in $H^\ast(X_D)$.
\end{lem}
\begin{proof}
    By Lemma~\ref{lem:type2-coho-1}, it is enough to check when $m_1\leq m_2,m_3$.
    Suppose that $(px_1+qx_2+rx_3)^{m_2} =0$ in $H^\ast(X_D)$, where $(q,r)\neq (0,0)$. Then we have
    \[
    (px_1+qx_2+rx_3)^{m_2}=\alpha x_1^{m_1}+q^{m_2}x_2(x_1+x_2)^{m_2-1}+\gamma x_3(x_1+x_2+x_3)^{m_3-2}
    \] as polynomials. Here, we may assume that $q\neq 0$. Indeed, if $q=0$, then $\gamma=0$, so $r=0$. This is a contradiction to $(q,r)\neq (0,0)$. Since $m_1>2$, by comparing the coefficients of $x_1x_2^{m_2-1}$ and $x_1^2x_2^{m_2-2}$ in the above, respectively, we get
    \begin{align*}
        m_2 p q^{m_2-1}&=q^{m_2}(m_2-1),\\
        {m_2\choose 2}p^2 q^{m_2-2}&=q^{m_2}{m_2-1 \choose 2}
    \end{align*}
    Hence we have
    \[ m_2p=q(m_2-1)\quad \text{ and }\quad m_2p^2=q^2(m_2-2).
    \]
    Thus $(m_2-1)p=q(m_2-2)$. There is no integer $m_2$ satisfying both $m_2p=q(m_2-1)$ and $(m_2-1)p=q(m_2-2)$. Hence $(px_1+qx_2+rx_3)^{m_2} \neq 0$ in $H^\ast(X_D)$. From a similar computation, one can show that $(px_1+qx_2+rx_3)^{m_3}\neq 0$ if $(p,q)\neq (0,0).$
\end{proof}

Now, we are ready to prove that the cohomology ring $H^\ast(X_D)$ determines the type of $D$. 

\begin{prop}\label{prop:type}
Let $D=D^I(m_1,m_2,m_3)$ and $\widetilde{D}=D^{I\!I}(\tilde{m}_1,\tilde{m}_2,\tilde{m}_3)$.
Then $H^\ast(X_D)$ and $H^\ast(X_{\widetilde{D}})$ cannot be isomorphic as graded rings.
\end{prop}
\begin{proof}
    Suppose that $H^\ast(X_D) \cong H^\ast(X_{\widetilde{D}})$. 
    Then there exists a graded isomorphism $\varphi$ from $\Z[x_1,x_2,x_3]$  to $\Z[y_1,y_2,y_3]$ such that $\mathcal{I}_D$ sends to $\mathcal{I}_{\widetilde{D}}$. Then there is a matrix \(G=[g_{ij}]\in \GL_3(\mathbb{Z})\) satisfying that \(|\det G|=1\) and \([\varphi(x_1),\varphi(x_2),\varphi(x_3)]^{\T}=G[y_1,y_2,y_3]^{\T}\).
    It follows from Lemma~\ref{lem:type2-coho-3} that $\tilde{m}_1=2$. We may further assume that $m_1=2$. Then $m_2=\tilde{m}_2$ and the isomorphism $\varphi$ should send the set $\{\pm x_1,\pm x_2\}$ to $\{\pm y_1,\pm((m_2-1)y_1+m_2y_2)\}$. Then $|\det G|=|g_{33}|m_2$. Since $m_2\geq 2$, we get $|\det G|\neq 1$, which is a contradiction.
\end{proof}

Let us consider the three-dissections of the first type. 
We prepare one lemma about the cohomology of the toric varieties of the first type dissections.

\begin{lem}\label{lem:type1-coho}
    Let \(D=D^{I}(m_1,m_2,m_3)\). There is no element $px_1+qx_2+rx_3\in H^2(X_D)$ with $r\neq 0$ such that $(px_1+qx_2+rx_3)^{m_3}=0$ in $H^\ast(X_D)$.
\end{lem}
\begin{proof}
    Suppose that $(px_1+qx_2+rx_3)^{m_3}=0$ in $H^\ast(X_D)$ and $r\neq 0$. Then
    \begin{equation}\label{eq:type1-coho-1}
        (px_1+qx_2+rx_3)^{m_3}=\alpha x_1^{m_1} + \beta x_2^{m_2} + r^{m_3}x_3(-x_1+x_2+x_3)(x_2+x_3)^{m_3-2}
    \end{equation}
    as polynomials, where $\alpha$ and $\beta$ are homogeneous polynomials of degree $m_3-m_1$ and $m_3-m_2$, respectively. Note that if $m_3-m_1$ and $m_3-m_2$ are negative, then $\alpha$ and $\beta$ are zero, respectively. Since we assume $r\neq 0$, neither $p$ nor $q$ is zero. However, this is impossible. If $m_1>2$, then we get $p=0$ by comparing the coefficients of
    $x_1x_3^{m_3-1}$ and 
    $x_1^2x_3^{m_3-2}$ in~\eqref{eq:type1-coho-1}. 
    If $m_2>2$, then we get $q=0$ by comparing the coefficients of $x_2x_3^{m_3-1}$ and $x_2^2x_3^{m_3-2}$ in~\eqref{eq:type1-coho-1}.
    If \(m_1=m_2=2\), then by comparing the coefficients of \(x_1x_3^{m_3-1}\), \(x_2x_3^{m_3-1}\), and \(x_1x_2x_3^{m_3-2}\), we get \(m_3p=-r\), and \(m_3q=r(m_3-1)\), and \(m_3(m_3-1)pq=-r^2(m_3-2)\). After doing some calculation, we get \(q=0\).
\end{proof}

\begin{prop}\label{prop:type1}
Let $D=D^I(m_1,m_2,m_3)$ and $\widetilde{D}=D^{I}(\tilde{m}_1,\tilde{m}_2,\tilde{m}_3)$. If $H^\ast(X_D)\cong H^\ast(X_{\widetilde{D}})$, then $D=\widetilde{D}$.
\end{prop}
\begin{proof}
We may assume that \(m_1\leq m_2\) and \(\widetilde{m}_1\leq \widetilde{m}_2\). From Theorem~\ref{thm:main}, we have
\begin{align*}
H^*(X_{D})&=\mathbb{Z}[x_1,x_2,x_3]/\langle x_1^{m_1}, x_2^{m_2}, x_3(-x_1+x_2+x_3)(x_2+x_3)^{m_3-2} \rangle,\\    
H^*(X_{\widetilde{D}})&=\mathbb{Z}[y_1,y_2,y_3]/\langle y_1^{\widetilde{m}_1}, y_2^{\widetilde{m}_2}, y_3(-y_1+y_2+y_3)(y_2+y_3)^{\widetilde{m}_3-2} \rangle.
\end{align*}

By assuming $(m_1,m_2,m_3)\neq (\tilde{m}_1,\tilde{m}_2,\tilde{m}_3)$, we show that $H^\ast(X_D)\not\cong H^\ast(X_{\tilde{D}})$.
We consider two cases according to the set of out-degrees.

\textbf{Case 1.} (\(\{m_1,m_2,m_3\}=\{\widetilde{m}_1,\widetilde{m}_2, \widetilde{m}_3\}=\{a,a,b\}\) with \(a \neq b\))\\
We may assume that \(m_3=b\) and \(\widetilde{m}_3=a\). Then \(n(\widetilde{D})=1\) while \(n(D)=2\) if \(a<b\), and \(n(D)=0\) otherwise. It follows that \(H^*(X_D)\not\cong H^*(X_{\widetilde{D}})\).

\textbf{Case 2.} (\(\{m_1,m_2,m_3\}=\{\widetilde{m}_1,\widetilde{m}_2, \widetilde{m}_3\}=\{a,b,c\}\) with \(a<b<c\))\\
In this case, $(m_1,m_2)$ is $(a,b)$, $(a,c)$, or $(b,c)$. Hence \(m_1=\widetilde{m}_1\) if and only if \(n(D)=n(\widetilde{D})\). Thus it suffices to check when \(m_3=b\) and \(\widetilde{m}_3=c\). Then \(H^*(X_{\widetilde{D}})\) has a linear term \(\alpha\) such that \(\alpha^b=0\), but \(H^*(X_{D})\) has no such term by Lemma~\ref{lem:type1-coho}. It follows that \(H^*(X_D)\not\cong H^*(X_{\widetilde{D}})\).

In any case, we conclude that \(H^*(X_D)\not\cong H^*(X_{\widetilde{D}})\) as desired.
\end{proof}

Finally, we consider the three-dissections of the second type.

\begin{prop}\label{prop:type2}
Let $D=D^{I\!I}(m_1,m_2,m_3)$ and $\widetilde{D}=D^{I\!I}(\tilde{m}_1,\tilde{m}_2,\tilde{m}_3)$. If  $H^\ast(X_D)\cong H^\ast(X_{\widetilde{D}})$, then $D=\widetilde{D}$.
\end{prop}
\begin{proof}
If $H^\ast(X_D)\cong H^\ast(X_{\widetilde{D}})$, there exists a graded isomorphism $\varphi$ from $\Z[x_1,x_2,x_3]$  to $\Z[y_1,y_2,y_3]$ such that $\mathcal{I}_D$ sends to $\mathcal{I}_{\widetilde{D}}$. Then there is a matrix \(G=[g_{ij}]\in \GL_3(\mathbb{Z})\) satisfying that \(|\det G|=1\) and \([\varphi(x_1),\varphi(x_2),\varphi(x_3)]^{\T}=G[y_1,y_2,y_3]^{\T}\). In the following, we show that the assumptions $(m_1,m_2,m_3)\neq (\tilde{m}_1,\tilde{m}_2,\tilde{m}_3)$ and $H^\ast(X_D)\cong H^\ast(X_{\widetilde{D}})$ lead to a contradiction to the existence of a matrix $G$.

Now we consider two cases according to the set of out-degrees. 

\textbf{Case 1.} (\(\{m_1,m_2,m_3\}=\{\widetilde{m}_1,\widetilde{m}_2, \widetilde{m}_3\}=\{a,a,b\}\) with \(a \neq b\))\\
In order to have that \(n(D)=n(\widetilde{D})\), we set \(m_1=\widetilde{m}_1\). We may assume that \((m_1,m_2,m_3)=(a,b,a)\) and \((\widetilde{m}_1,\widetilde{m}_2,\widetilde{m}_3)=(a,a,b)\). Then we have
\begin{align*}
\mathcal{I}_D&=\langle x_1^{a}, x_2(x_1+x_2)^{b-1}, x_3(x_1+x_2+x_3)^{a-1} \rangle,\\ 
\mathcal{I}_{\widetilde{D}}&=\langle y_1^{a}, y_2(y_1+y_2)^{a-1}, y_3(y_1+y_2+y_3)^{b-1} \rangle.
\end{align*}

We divide two sub-cases according to the size of $a$ and $b$.

\textbf{Case 1-1.} (\(a<b\))

Since $a<b$, \(\varphi(x_1^a)=\pm y_1^a\). It follows that 
\begin{align}
\label{eq:case2-1-1-2}
g_{12}=g_{13}&=0,\\
\label{eq:case2-1-1-3}
|g_{11}|=1, & \quad
|g_{22}g_{33}-g_{23}g_{32}|=1.
\end{align}
Since \(\varphi(x_3(x_1+x_2+x_3)^{a-1})\) and \(\varphi(x_2(x_1+x_2)^{b-1})\) belong to $\mathcal{I}_{\widetilde{D}}$, it follows from~\eqref{eq:case2-1-1-2} that
\begin{equation}\label{eq:case2-1a}
\begin{split}
(g_{31}y_1+g_{32}y_2+g_{33}y_3)\{(g_{11}+g_{21}+g_{31})y_1+(g_{22}+g_{32})y_2
+(g_{23}+g_{33})y_3\}^{a-1}\\
=g_{31}(g_{11}+g_{21}+g_{31})^{a-1}y_1^a+g_{32}(g_{22}+g_{32})^{a-1}y_2(y_1+y_2)^{a-1}
\end{split}
\end{equation}
and
\begin{equation}\label{eq:case2-1b}
\begin{split}
(g_{21}y_1+g_{22}y_2+g_{23}y_3)\{(g_{11}+g_{21})y_1+g_{22}y_2+g_{23}y_3\}^{b-1}\\
=\alpha y_1^a+\beta y_2(y_1+y_2)^{a-1}+g_{23}^b y_3(y_1+y_2+y_3)^{b-1}
\end{split}
\end{equation}
as polynomials, where $\alpha$ and $\beta$ are homogeneous polynomials of degree $b-a$.
Comparing the coefficients of \(y_3^a\) in~\eqref{eq:case2-1a}, we have 
\(g_{33}(g_{23}+g_{33})=0\). From this fact together with \eqref{eq:case2-1-1-3}, it follows that 
\begin{equation}\label{eq:case2-1-1-5}
|g_{23}|=1.
\end{equation}
Comparing the coefficients of \(y_2y_3^{b-1}\) in both sides of~\eqref{eq:case2-1b}, we have \(bg_{22}=(b-1)g_{23}\). By \eqref{eq:case2-1-1-5}, we get \(|g_{22}|=(b-1)/b\), which contradicts to \(g_{22}\in\mathbb{Z}\). Therefore, we conclude that there is no such \(G\).

\textbf{Case 1-2.} (\(b<a\))

Since \(\varphi(x_1^{a})\) belongs to $\mathcal{I}_{\widetilde{D}}$, we have 
\begin{equation}\label{eq:case2-1-2a}
(g_{11}y_1+g_{12}y_2+g_{13}y_3)^a=g_{11}^ay_1^a+g_{12}^ay_2(y_1+y_2)^{a-1}+\gamma y_3(y_1+y_2+y_3)^{b-1},
\end{equation}
where $\gamma$ is a homogeneous polynomial of degree $a-b$.
Comparing the coefficients of \(y_1y_2^{a-1}\) and \(y_1^{a-1}y_2\) in both sides of~\eqref{eq:case2-1-2a}, we obtain 

\begin{align}
\label{eq:case2-1-2-1}
ag_{11}g_{12}^{a-1}&=(a-1)g_{12}^a,\\
\label{eq:case2-1-2-2}
ag_{11}^{a-1}g_{12}&=g_{12}^a.
\end{align}

From \eqref{eq:case2-1-2-1}, we get \(g_{12}=0\) or \(ag_{11}=(a-1)g_{12}\). However, if \(g_{12}\neq 0\), then \(g_{11}=g_{12}\) yields that \(a=2\) by \eqref{eq:case2-1-2-2}, which contradicts to \(a>2\). Thus, we have
    $g_{12}=0.$
Applying this to~\eqref{eq:case2-1-2a}, we get $\gamma=0$ since there is no term containing $y_2$ on the left-hand side of~\eqref{eq:case2-1-2a}. Thus $g_{13}=0$. Then $g_{11}\neq 0$ since $\det G\neq 0$.
Since $g_{12}=g_{13}=0$ and \(\varphi(x_2(x_1+x_2)^{b-1})\) belongs to $\mathcal{I}_{\widetilde{D}}$, we have
\begin{equation}\label{eq:case2-1-2b}
\begin{split}
(g_{21}y_1+g_{22}y_2+g_{23}y_3)\{(g_{11}+g_{21})y_1+g_{22}y_2+g_{23}y_3\}^{b-1}\\
=g_{23}^{b}y_3(y_1+y_2+y_3)^{b-1}.
\end{split}
\end{equation}
Comparing the coefficients of \(y_2^b\) in both sides of~\eqref{eq:case2-1-2b}, we get 
    $g_{22}=0.$
Then there is no term containing $y_2$ on the left side of~\eqref{eq:case2-1-2b}, so $g_{23}=0$. This is a contradiction to $\det G\neq 0$.

\textbf{Case 2.} (\(\{m_1,m_2,m_3\}=\{\widetilde{m}_1,\widetilde{m}_2, \widetilde{m}_3\}=\{a,b,c\}\) with \(a<b<c\))

We consider three sub-cases.

\textbf{Case 2-1.} (\((m_1,m_2,m_3)=(a,b,c)\) and \((\widetilde{m}_1,\widetilde{m}_2,\widetilde{m}_3)=(a,c,b)\))\\
In this case we have
\begin{align*}
H^*(X_{D})&=\mathbb{Z}[x_1,x_2,x_3]/\langle x_1^{a}, x_2(x_1+x_2)^{b-1}, x_3(x_1+x_2+x_3)^{c-1} \rangle,\\ 
H^*(X_{\widetilde{D}})&=\mathbb{Z}[y_1,y_2,y_3]/\langle y_1^{a}, y_2(y_1+y_2)^{c-1}, y_3(y_1+y_2+y_3)^{b-1} \rangle.
\end{align*}
Since $a<b<c$, we have \(\varphi(x_1^a)=\pm y_1^a\), so \(|g_{11}|=1\) and \(g_{12}=g_{13}=0\). Since \(\varphi(x_2(x_1+x_2)^{b-1})\) belongs to $\mathcal{I}_{\widetilde{D}}$, we have
\[
(g_{21}y_1+g_{22}y_2+g_{23}y_3)\{(g_{11}+g_{21})y_1+g_{22}y_2+g_{23}y_3\}^{b-1}=\alpha y_1^a +g_{23}^b y_3 (y_1+y_2+y_3)^{b-1},
\] where $\alpha$ is a homogeneous polynomial of degree $b-a$.
Comparing the coefficients of~$y_2^b$, we have \(g_{22}=0\). However, this yields that \(g_{23}=0\) so that \(\det G=0\), which is a contraction. There is no such \(G\).

\textbf{Case 2-2.} (\((m_1,m_2,m_3)=(b,a,c)\) and \((\widetilde{m}_1,\widetilde{m}_2,\widetilde{m}_3)=(b,c,a)\))
In this case, we have
\begin{align*}
H^*(X_{D})&=\mathbb{Z}[x_1,x_2,x_3]/\langle x_1^{b}, x_2(x_1+x_2)^{a-1}, x_3(x_1+x_2+x_3)^{c-1} \rangle,\\ 
H^*(X_{\widetilde{D}})&=\mathbb{Z}[y_1,y_2,y_3]/\langle y_1^{b}, y_2(y_1+y_2)^{c-1}, y_3(y_1+y_2+y_3)^{a-1} \rangle.
\end{align*}
Since \(\varphi(x_1^b)\) belongs to $\mathcal{I}_{\widetilde{D}}$, we have
\begin{equation}\label{eq:case2-2-2b}
(g_{11}y_1+g_{12}y_2+g_{13}y_3)^b=g_{11}^by_1^b+\gamma y_3(y_1+y_2+y_3)^{a-1},
\end{equation}
where $\gamma$ is a homogeneous polynomial of degree $b-a$.
Comparing the coefficients of \(y_2^b\) in both sides of \eqref{eq:case2-2-2b}, we get
    $g_{12}=0.$
Then there is no term having $y_2$ on the left side of \eqref{eq:case2-2-2b}, so $\gamma$ should be zero. Thus $|g_{11}|=1$ and $g_{13}=0$.
Since \(\varphi(x_2(x_1+x_2)^{a-1})\) belong to $\mathcal{I}_{\widetilde{D}}$, we have
\begin{equation}\label{eq:case2-2-2a}
\begin{split}
(g_{21}y_1+g_{22}y_2+g_{23}y_3)\{(g_{11}+g_{21})y_1+g_{22}y_2+g_{23}y_3\}^{a-1}\\
=g_{23}^{a}y_3(y_1+y_2+y_3)^{a-1}.
\end{split}
\end{equation}
Comparing the coefficients of \(y_2^a\) in both sides of \eqref{eq:case2-2-2a}, we have \(g_{22}=0\). Then there is no term having \(y_2\) on the left-hand side of \eqref{eq:case2-2-2a}, so $g_{23}=0$. This is a contradiction to $\det G\neq 0$. There is no such \(G\).

\textbf{Case 2-3.} (\((m_1,m_2,m_3)=(c,a,b)\) and \((\widetilde{m}_1,\widetilde{m}_2,\widetilde{m}_3)=(c,b,a)\))

In this case we have
\begin{align*}
H^*(X_{D})&=\mathbb{Z}[x_1,x_2,x_3]/\langle x_1^{c}, x_2(x_1+x_2)^{a-1}, x_3(x_1+x_2+x_3)^{b-1} \rangle,\\ 
H^*(X_{\widetilde{D}})&=\mathbb{Z}[y_1,y_2,y_3]/\langle y_1^{c}, y_2(y_1+y_2)^{b-1}, y_3(y_1+y_2+y_3)^{a-1} \rangle.
\end{align*}
Since \(\varphi(x_2(x_1+x_2)^{a-1})\) belongs to $\mathcal{I}_{\widetilde{D}}$, we have
\begin{equation}\label{eq:case2-2-3a}
\begin{split}
(g_{21}y_1+g_{22}y_2+g_{23}y_3)\{(g_{11}+g_{21})y_1+(g_{12}+g_{22})y_2+(g_{13}+g_{23})y_3\}^{a-1}\\
=g_{23}(g_{13}+g_{23})^{a-1}y_3(y_1+y_2+y_3)^{a-1}.
\end{split}
\end{equation}
Here, \(g_{23}(g_{13}+g_{23})\neq0\). If not, it contradicts to $\det G\neq 0$.

Comparing the coefficients of \(y_1^a\) and \(y_2^a\) in both sides of \eqref{eq:case2-2-3a}, we have
\begin{equation}\label{eq:case2-2-3-1}
g_{21}(g_{11}+g_{21})=0 \quad \text{and} \quad g_{22}(g_{12}+g_{22})=0.
\end{equation}
There are four possibilities.
\begin{enumerate}
    \item[{\rm i)}] If \(g_{21}=0\) and \(g_{22}\neq0\), then \eqref{eq:case2-2-3a} becomes
    \[
    (g_{22}y_2+g_{23}y_3)\{g_{11}y_1+(g_{13}+g_{23})y_3\}^{a-1}=g_{23}(g_{13}+g_{23})^{a-1}y_3(y_1+y_2+y_3)^{a-1}.
    \]
    Comparing the coefficients of \(y_1^{a-1}y_2\) in both sides of the above equation, we have \(g_{11}=0\). Then there is no term having $y_1$ on the left-hand side, which is a contradiction. 
    \item[{\rm ii)}] If \(g_{21}\neq0\) and \(g_{22}=0\), then \eqref{eq:case2-2-3a} becomes
    \[
    (g_{21}y_1+g_{23}y_3)\{g_{12}y_2+(g_{13}+g_{23})y_3\}^{a-1}=g_{23}(g_{13}+g_{23})^{a-1}y_3(y_1+y_2+y_3)^{a-1}.
    \]
    Comparing the coefficients of \(y_1y_2^{a-1}\) in both sides of the above equation, we have \(g_{12}=0\). Then there is no term having $y_2$ on the left-hand side, which is a contradiction. 
    \item[{\rm iii)}] If \(g_{21}=g_{22}=0\), then \eqref{eq:case2-2-3a} becomes
    \[
    y_3\{g_{11}y_1+g_{12}y_2+(g_{13}+g_{23})y_3\}^{a-1}=(g_{13}+g_{23})^{a-1}y_3(y_1+y_2+y_3)^{a-1},
    \]
    so we have
    \begin{equation*}\label{eq:case2-2-3-3}
    g_{11}=g_{12}=g_{13}+g_{23}.
    \end{equation*}
    From the fact that \(\varphi(x_3(x_1+x_2+x_3)^{b-1})\) belongs to $\mathcal{I}_{\widetilde{D}}$, we have
\begin{equation}\label{eq:cases2-2-3-4}
    \begin{split}
        (g_{31}y_1+g_{32}y_2+g_{33}y_3)\{(g_{11}+g_{31})y_1 +(g_{12}+g_{32})y_2+(g_{13}+g_{23}+g_{33})y_3\}^{b-1}\\
        =g_{32}(g_{12}+g_{23})^{b-1}y_2(y_1+y_2)^{b-1}+\alpha y_3(y_1+y_2+y_3)^{a-1},
    \end{split}
\end{equation}
where $\alpha$ is a homogeneous polynomial of degree $b-a$.
Plugging $y_1=1$, $y_2=-1$, and $y_3=0$ into~\eqref{eq:cases2-2-3-4}, we get $g_{31}=g_{32}$. This contradicts to $\det G\neq 0$.
    \item[{\rm iv)}] If \(g_{21}\neq0\) and \(g_{22}\neq0\), then by \eqref{eq:case2-2-3-1}, 
    we get
    \[g_{11}+g_{21}=0\quad \text{and} \quad g_{12}+g_{22}=0.\]
    Hence
     \eqref{eq:case2-2-3a} becomes
    \begin{equation*}
            (g_{21}y_1+g_{22}y_2+g_{23}y_3)y_3^{a-1}
            =g_{23}y_3(y_1+y_2+y_3)^{a-1}.
    \end{equation*}
    On the left-hand side of the above equation, the degree of $y_1$ or $y_2$ in each term is at most one, so $a=2$, and $g_{21}=g_{22}=g_{23}$. Since $g_{11}+g_{21}=0$ and $g_{12}+g_{22}=0$, we get
    \begin{equation*}\label{eq:case2-2-3-2}
        g_{11}=g_{12}=-g_{21}=-g_{22}=-g_{23}\neq0.
    \end{equation*}
From the fact that \(\varphi(x_3(x_1+x_2+x_3)^{b-1})\) belongs to $\mathcal{I}_{\widetilde{D}}$, we have
\begin{equation}\label{eq:cases2-2-3-3}
    \begin{split}
        (g_{31}y_1+g_{32}y_2+g_{33}y_3)\{g_{31}y_1 +g_{32}y_2+(g_{13}+g_{23}+g_{33})y_3)\}^{b-1}\\
        =g_{32}g_{23}^{b-1}y_2(y_1+y_2)^{b-1}+\alpha y_3(y_1+y_2+y_3)^{a-1},
    \end{split}
\end{equation}
where $\alpha$ is a homogeneous polynomial of degree $b-a$.
Plugging $y_1=1$, $y_2=-1$, and $y_3=0$ into~\eqref{eq:cases2-2-3-3}, we get \(g_{31}=g_{32}\), which is a contradiction to $\det G\neq 0$.
\end{enumerate} 
Hence, in any case, there is no isomorphism between $H^\ast(X_D)$ and $H^\ast(X_{\widetilde{D}})$.
\end{proof}

Combining Propositions~\ref{prop:type},~\ref{prop:type1}, and~\ref{prop:type2}, we prove Theorem~\ref{thm1}.

\section{Further discussion}\label{sec:further}

In this section, we discuss Problem~\ref{prob1} and introduce a problem about a relation between toric varieties of Schr\"{o}der type and flag varieties.

Let $D$ be a $k$-dissection of $\mathsf{P}_{n+2}$. Using Corollary~\ref{cor:leaf}, we define a set $L(D)$ by the set of leaves $w$ in $\tau(D)$ such that $\phi(w)=\{a_i,b_i\}$ in Theorem~\ref{thm:main}. Then for every $w\in L(D)$, if $v$ is the parent of $w$, then $x_{a_ib_i}^{\ell_i}=0$ in $H^\ast(X_D)$, where $\ell_i$ is the number of children of $v$.

The following proposition gives another evidence for the affirmative answer to Problem~\ref{prob1}.

\begin{prop}\label{prop:further}
Let $\tau(D)$ and $\tau(\tilde{D})$ be the Schr\"{o}der trees with at least four internal vertices such that all the internal vertices have the same number of children. If $|L(D)|\neq |L(\tilde{D})|$, then $H^\ast(X_D)\not\cong H^\ast(X_{\widetilde{D}})$. 
\end{prop}
\begin{proof}
Assume that each vertex of $\tau(D)$ and $\tau(\tilde{D})$ has $\ell$ children. 
If $\ell=2$, then $X_D$ and $X_{\widetilde{D}}$ are toric varieties of Catalan type. It is shown in~\cite{LMP_directed_Dynkin} that two toric varieties of Catalan type are isomorphic if and only their integral cohomology rings are isomorphic as graded rings. Thus it is enough to prove the case when $\ell>2$.

Suppose that $|L(D)|<|L(\widetilde{D})|$ and $H^\ast(X_D)\cong H^\ast(X_{\widetilde{D}})$. Then there exists a non-monomial element $\alpha=\sum_{i=1}^kc_ix_{a_ib_i}$ in $H^2(X_D)$ such that $\alpha^{\ell}=0$ in $H^\ast(X_D)$. Let $v_{i_1},v_{i_2},v_{i_3}$ be three adjacent internal vertices of $\tau(D)$. Consider the Schr\"{o}der subtree $T$ of $\tau(D)$ by taking $v_{i_1},v_{i_2},v_{i_3}$ and their children. Then $T$ defines a toric variety $X_T$ and $H^\ast(X_T)$ is obtained from $H^\ast(X_D)$ by substituting $x_{a_ib_i}=0$ whenever the parent of the vertex labeled by $\{a_i,b_i\}$ is not in $\{v_{i_1},v_{i_2},v_{i_3}\}$. Since the restriction $\alpha_T$ of $\alpha$ to $H^\ast(X_T)$ also satisfies that $\alpha^\ell=0$, Lemmas~\ref{lem:type2-coho-3} and~\ref{lem:type1-coho} imply that at most one of $c_{i_1}$, $c_{i_2}$, and $c_{i_3}$ is nonzero. Hence if $c_i$ and $c_{i'}$ are nonzero, then the distance between $v_i$ and $v_{i'}$ is greater than two. 
Let $\mathcal{H}$ be the subring of $H^\ast(X_D)$ obtained by substituting $x_{a_jb_j}=0$ for $j\neq i,i'$.
Then $(c_ix_{a_ib_i}+c_{i'}x_{a_{i'}b_{i'}})^\ell=0$ in $\mathcal{H}$, where $\mathcal{H}$ is
$$\Z[x_{a_ib_i},x_{a_{i'}b_{i'}}]/\langle x_{a_ib_i}^\ell,x_{a_{i'}b_{i'}}(\pm x_{a_ib_i}+x_{a_{i'}b_{i'}})^{\ell-1}\rangle$$
or $$\Z[x_{a_ib_i},x_{a_{i'}b_{i'}}]/\langle x_{a_ib_i}^\ell,x_{a_{i'}b_{i'}}^{\ell}\rangle.$$
Since $\ell>2$, in any case, there is no linear element $c_ix_{a_ib_i}+c_{i'}x_{a_{i'}b_{i'}}$ such that $c_ic_{i'}\neq 0$ and $(c_ix_{a_ib_i}+c_{i'}x_{a_{i'}b_{i'}})^\ell=0$ in $\mathcal{H}$.
Therefore, there is no non-monomial element $\alpha\in H^2(X_D)$ such that $\alpha^\ell=0$ in $H^\ast(X_D)$. This proves the proposition.
\end{proof}

We close this section by introducing a problem about a relation between toric varieties of Schr\"{o}der type and flag varieties. It is shown in~\cite{LMP_Catalan} that every toric variety of Catalan type appears as a certain Richardson variety in the full flag variety. Since the family of toric varieties of Schr\"{o}der type includes the toric varieties of Catalan type, one can ask naturally whether toric varieties of Schr\"{o}der type can be realized as subvarieties of flag variety like Richardson varieties.

\begin{problem}
For each dissection $D$ of a polygon, does the toric variety $X_D$ appear in a (partial) flag variety as a Richardson variety?
\end{problem}


\end{document}